\documentclass{article}

\usepackage{myarxiv}
\usepackage[utf8]{inputenc} 
\usepackage[T1]{fontenc}    
\usepackage{hyperref}       
\usepackage{url}            
\usepackage{booktabs}       
\usepackage{amsfonts}       
\usepackage{nicefrac}       
\usepackage{microtype}      
\usepackage{lipsum}

\usepackage{amssymb,amsmath,amsfonts,amsthm,enumerate}
\usepackage{subfigure}
\usepackage{graphicx}
\usepackage{epsfig}
\usepackage{float}
\usepackage{cite}
\usepackage{bm}
\usepackage{enumitem}
\usepackage{empheq}
\usepackage{dsfont}
\setlist{nosep}

\hypersetup{
  colorlinks   = true, 
  urlcolor     = black, 
  linkcolor    = black, 
  citecolor    = black 
}

\newtheorem{thm}{\indent Theorem}[section]
\newtheorem{lem}[thm]{\indent Lemma}

\newtheorem{coro}[thm]{\indent Corollary}

\newtheorem{prob}[thm]{\indent Problem}
\newtheorem{rem}[thm]{\indent Remark}

\newtheorem{ex}[thm]{\indent Example}

\numberwithin{equation}{section}
\allowdisplaybreaks[4]
\def\BC{\mathbb C}
\def\BN{\mathbb N}
\def\BR{\mathbb R}

\def\cD{\mathcal D}

\def\rd{\mathrm d}

\def\rdiv{\mathrm{div}}
\def\e{\mathrm e}

\def\rspan{\mathrm{span}}

\def\Ga{\Gamma}

\def\La{\Lambda}
\def\Si{\Sigma}
\def\Om{\Omega}
\def\al{\alpha}
\def\be{\beta}
\def\ga{\gamma}
\def\de{\delta}

\def\ve{\varepsilon}
\def\te{\theta}

\def\ka{\kappa}
\def\la{\lambda}
\def\si{\sigma}
\def\vp{\varphi}
\def\om{\omega}
\def\f{\frac}
\def\nb{\nabla}
\def\ov{\overline}
\def\pa{\partial}
\def\wh{\widehat}
\def\wt{\widetilde}

\def\sumn{\sum_{n=1}^{\infty}}

\title{Uniqueness of orders and parameters in multi-term\\
time-fractional  diffusion equations\\
by short-time behavior}

\author{
  Yikan \uppercase{Liu}\\
  Research Center of Mathematics for Social Creativity, Research Institute for Electronic Science,\\
  Hokkaido University, N12W7, Kita-Ward, Sapporo 060-0812, Japan.\\
  \texttt{ykliu@es.hokudai.ac.jp}\\
  \And
  Masahiro \uppercase{Yamamoto}\\
  Graduate School of Mathematical Sciences, The University of Tokyo\\
  3-8-1 Komaba, Meguro-ku, Tokyo 153-8914, Japan;\\
  Honorary Member of Academy of Romanian Scientists\\
  Ilfov, nr. 3, Bucuresti, Romania;\\
  Correspondence member of Accademia Peloritana dei Pericolanti, Palazzo Universit\`a\\
  Piazza S. Pugliatti 1 98122 Messina, Italy.\\
  \texttt{myama@next.odn.ne.jp}\\
}
\newcommand{\shortauthor}{Y. Liu, M. Yamamoto}

\begin{document}
\maketitle

\begin{abstract}
As the most significant difference from parabolic equations, 
long-time or short-time behavior of solutions to time-fractional evolution 
equations is dominated by the fractional orders, whose unique determination 
has been frequently investigated in literature. Unlike all the existing 
results, in this article we prove the uniqueness of orders and parameters 
(up to a multiplier for the latter) only by principal terms of asymptotic 
expansions of solutions near $t=0$ at a single spatial point. 
Moreover, we discover special conditions on unknown initial values or source terms for the coincidence of observation data. As a byproduct, we can even conclude the uniqueness for initial values or source terms by the same data. The proof relies on the asymptotic expansion after taking the Laplace transform and the completeness of generalized eigenfunctions.
\end{abstract}

\MRsubject{35R30\and 35R11\and 58J99}

\section{Introduction}\label{sec-intro}

Let $T\in\BR_+:=(0,\infty)$ and $\Om\subset\BR^d$ ($d\in\BN=\{1,2,\ldots\}$) be a bounded domain with smooth boundary $\pa\Om$. For a constant $m\in\BN$, let $\al_j,q_j$ ($j=1,\ldots,m$) be positive constants such that $1>\al_1>\al_2>\cdots>\al_m>0$. We define an elliptic operator $L:\cD(L):=H^2(\Om)\cap H_0^1(\Om)\longrightarrow L^2(\Om)$ by
\[
L h(\bm x):=-\rdiv(\bm a(\bm x)\nb h(\bm x))+\bm b(\bm x)\cdot\nb h(\bm x)+c(\bm x)h(\bm x),\quad\bm x\in\Om,
\]
where $\cdot$ and $\nb=(\f\pa{\pa x_1},\ldots,\f\pa{\pa x_d})$ refer to the inner product in $\BR^d$ and the gradient in $\bm x$, respectively. We emphasize that $L$ is not necessarily symmetric. Here
\begin{equation}\label{eq-cond-L1}
\bm a=(a_{i j})_{1\le i,j\le d}\in C^1(\ov\Om\,;\BR^{d\times d}),\quad\bm b=(b_1,\ldots,b_d)\in C^1(\ov\Om\,;\BR^d),\quad c\in C(\ov\Om)
\end{equation}
are real-valued $\bm x$-dependent matrix-, vector- and scalar-valued functions, respectively. In particular, we assume that $\bm a$ is symmetric and strictly positive-definite on $\ov\Om$\,, i.e., there exists a constant $\de>0$ such that
\begin{equation}\label{eq-cond-L2}
a_{i j}(\bm x)=a_{j i}(\bm x),\quad\bm a(\bm x)\bm\xi\cdot\bm\xi\ge\de|\bm\xi|^2,\quad\forall\,i,j=1,\ldots,d,\ \forall\,\bm x\in\ov\Om\,,\ \forall\,\bm\xi\in\BR^d,
\end{equation}
where $|\bm\xi|^2=\bm\xi\cdot\bm\xi$.

In this article, we are concerned with the following initial-boundary value problem for a multi-term time-fractional diffusion equation
\begin{equation}\label{eq-ibvp-u}
\left\{\begin{alignedat}{2}
& \!\!\left(\sum_{j=1}^m q_j\pa_t^{\al_j}+L\right)u(\bm x,t)=\rho(t)f(\bm x), & \quad & (\bm x,t)\in\Om\times(0,T),\\
& u(\bm x,0)=a(\bm x), & \quad & \bm x\in\Om,\\
& u(\bm x,t)=0, & \quad & (\bm x,t)\in\pa\Om\times(0,T).
\end{alignedat}\right.
\end{equation}
Here $\pa_t^{\al_j}$ denotes the $\al_j$-th order Caputo derivative in time, which is traditionally defined by (e.g. Podlubny \cite{P99})
\[
\pa_t^{\al_j}h(t):=\f1{\Ga(1-\al_j)}\int_0^t\f{h'(s)}{(t-s)^{\al_j}}\,\rd s,\quad h\in C^1[0,\infty).
\]
Here $\Ga(\,\cdot\,)$ stands for the Gamma function. Throughout this article, we basically assume $\rho\in L^1(0,T)$ and for convenience, we understand $\rho\in L^1(\BR_+)$ by identifying it with its zero extension in $(T,\infty)$.

The governing equation in \eqref{eq-ibvp-u} is called a single-term 
time-fractional diffusion equation if $m=1$, while it is called a multi-term 
one if $m\ge2$. In the last decades, they have gathered increasing population 
as typical nonlocal models for anomalous diffusion in heterogeneous media or 
fractals (e.g. \cite{BP88,HH98}). Especially, linear theories for 
the case of $m=1$ have been established in recent years.
Among many references, we refer only to \cite{SY11,GLY15,KRY20} and 
see the references therein. Compared with classical parabolic equations, 
single-term time-fractional diffusion equations inherit the same 
time-analyticity and the maximum principle and see e.g. \cite{SY11,LY19}. 
Meanwhile, they also share quantitive similarities in the smoothing effect and the unique continuation property (e.g. \cite{SY11,JLLY17}). In contrast, they 
show essential difference from their integer counterpart in the  
asymptotic behavior as $t \to \infty$, which indicates the rate 
$t^{-\al}$ (e.g. \cite{SY11}).

Though less studied, it reveals that the case of $m\ge2$ in \eqref{eq-ibvp-u} 
also behaves similarly to that of $m=1$. We refer to \cite{LLY15} for the 
well-posedness and the limited smoothing effect of solutions, and 
\cite{L11,L17} for weak and strong positivity properties, respectively. 
In \cite{BY13,LHY20a}, the parameters $q_j$ in \eqref{eq-ibvp-u} are 
generalized to be $\bm x$-dependent at the cost of slightly weaker solution 
regularity. Remarkably, the short- and long-time asymptotic behavior of 
solutions depend on the highest order $\al_1$ and the lowest one $\al_m$, respectively (e.g. \cite{LLY15,LHY20a}).

Since the orders of fractional derivatives are closely related to 
the heterogeneity and related physical properties of media, the determination of $\al_j$ in \eqref{eq-ibvp-u} turns out to be significant in both theory and practice. For the single-term case in one-dimensional case, 
Cheng et al. \cite{CNYY09} first proved the uniqueness of $\al_1$ by the data $u|_{\{\bm x_0\}\times(0,T)}$ with $\bm x_0\in\pa\Om$. Employing the above mentioned asymptotic behavior, Hatano et al. \cite{HNWY13} deduced inversion formulae
\[
\al_1=-\lim_{t\to\infty}\f{t\,\pa_t u(\bm x_0,t)}{u(\bm x_0,t)}=\lim_{t\to0}\f{t\,\pa_t u(\bm x_0,t)}{u(\bm x_0,t)-a(\bm x_0)}
\]
with $\bm x_0\in\Om$. For the multi-term case, Li and Yamamoto \cite{LY15} showed the unique determination of $m,\al_j,q_j$ with constants $q_j>0$ by $u|_{\{\bm x_0\}\times(0,T)}$ with $\bm x_0\in\Om$. For $\bm x$-dependent $q_j\ge0$, Li et al. \cite{LIY16} proved the unique determination of $m,\al_j,q_j$ and the zeroth-order coefficient $c$ using the Dirichlet-to-Neumann map, and Li et al. \cite{LHY20a} uniquely identified $\al_j$ by $u|_{\{\bm x_0\}\times(0,T)}$ with $\bm x_0\in\Om$. We refrain from providing a complete list, and refer to Li et al. \cite{LLY19} as a survey covering most literature on parameter inverse problems for \eqref{eq-ibvp-u} before 2019. We notice that due to the dimension of observation data, most literature dealt with the simultaneous identification of orders and unknown sources or coefficients in \eqref{eq-ibvp-u} 
(see \cite{LW19,LZ20,SLZ21,KLLY21,YZW21,Y21b}). 
Only \cite{LHY20b} obtained the stability in determining a single order by the data $u(\bm x_0,t_0)$. As recent progress, we refer to \cite{AU20,Y20} for the unique determination of time and space fractional orders, and Yamamoto \cite{Y21a} for treating a non-symmetric $L$. Very recently, Jin and Kian \cite{JK21a,JK21b} uniquely recovered $m,\al_j,q_j$ in \eqref{eq-ibvp-u} 
in an unknown medium, that is, $L$ is unknown.

Among the above literature on parameter inverse problems for \eqref{eq-ibvp-u}, the majority used the observation of $u$ at $\{\bm x_0\}\times(0,T)$. However, due to technical difficulties, almost all publications assume 
the symmetry of $L$, that is, $\bm b\equiv\bm0$, and only uniqueness was proved with data of solution such as $u(\bm x_0,t)$ at a monitoring point $\bm x_0\in\Om$
over a time interval. 
On the other hand, in the existing publications, 
since the observation data are a function of $t$ 
and there are only finite unknowns, such an inverse problem is overdetermined, which indicates the possibility to reduce the accuracy of data or determine another unknown ingredient in \eqref{eq-ibvp-u} simultaneously. 
Moreover, since the orders $\alpha_j$ greatly influence short-time behavior
of the solution, we can expect that principal terms of asymptotic expansions
of solutions near $t=0$ can characterize such orders, and we are motivated to 
establish the uniqueness for the following inverse problem.

\begin{prob}[Parameter inverse problem]\label{prob-pip}
Let $\bm x_0\in\Om$ and sufficiently small $\tau>0$ be fixed 
arbitrarily$,$ and let $u$ satisfy \eqref{eq-ibvp-u} with $f\equiv0$ or 
$a\equiv0$ in $\Om$. Determine $m$ and $\al_j,q_j$ for $j=1,\ldots,m$ 
simultaneously by asymptotic behavior of $u(\bm x_0,t)$ for $0<t<\tau$.
\end{prob}

In the formulation of Problem \ref{prob-pip}, as data we adopt 
the short-time behavior near $t=0$, in other words, principal terms 
of the asymptotic expansion in $t$ of the solution as $t\downarrow0$, which is
formulated by \eqref{eq-inexact} later.
We emphasize that we do not require any non-aymptotically observed quantities 
of the solution for the uniqueness of Problem \ref{prob-pip}, which is 
the main novelty of our article.

Moreover, the initial value $a$ and the spatial component $f$ of the source 
term in \eqref{eq-ibvp-u} are also assumed to be unknown. Finally, we mention 
that the elliptic part $L$ can be non-symmetric, which seems also seldom 
handled in references.

Only by the information of principal terms of asymptotic expansions
of solutions near $t=0$, we establish the uniqueness in determining $m$ and 
$\al_j$ in \eqref{eq-ibvp-u} simultaneously in the respective cases
$f \equiv 0$ and $a\equiv 0$.
Meanwhile, it reveals that $q_j$ can be uniquely identified up to a common 
multiplier depending on the local information of the unknown initial value 
$a$ or source term $f$, and the possible non-uniqueness is illustrated 
by some examples. 
Our uniqueness result can be regarded as the strong distinguishability 
effect of the fractional orders to the solution of \eqref{eq-ibvp-u}.

On this direction, it can be proved that 
unknown initial values or source terms should satisfy rather special 
relations in order to produce coinciding asymptotics. 
As a corollary, in a special case we can even conclude the uniqueness of the initial value $a$ or the spatial component $f$ of the source term
by data $u(\bm x_0,\,\cdot\,)$.
It is surprising because the one-dimensional data can uniquely determine a 
$d$-dimensional unknown. This is indeed an underdetermined problem and 
the data are sometimes called sparse data in literature (e.g. \cite{LZ20}).

The remainder of this paper is organized as follows. In Section \ref{sec-prelim}, we make preparations to state and prove the main results, especially the generalized eigenvalue theory of $L$. Next, we collect the main results concerning Problem \ref{prob-pip} along with illustrative examples and remarks in Section \ref{sec-main}. Then Section \ref{sec-proof} is devoted to the proofs of main results, followed by conclusions provided in Section \ref{sec-rem}.

\section{Preliminaries}\label{sec-prelim}

Throughout this article, by $C>0$ we denote generic constants which may change line by line. For $h\in L^1(\BR_+)$, we denote its Laplace transform by
\[
\wh h(p):=\int_{\BR_+}\e^{-p t}h(t)\,\rd t,\quad p>0.
\]
For $h\in L^1(\BR_+)$ and a constant $\mu>-1$, by the notation $h(t)\sim t^\mu$ for small $t>0$, we mean that there exists a constant $C'\ne0$ such that $\lim_{t\to0+}t^{-\mu}h(t)=C'$. In this case, we know $h(t)=O(t^\mu)$ but $h(t)\ne o(t^{\mu'})$ for any $\mu'>\mu$ a.e. as $t\to0+$. We define $\sim$ for a.e. 
large $t>0$ in the same manner.

In the sequel, the inner product of $L^2(\Om)$ is denoted by $(\,\cdot\,,\,\cdot\,)$. By $\si(L)=\{\la_n\}_{n\in\BN}\subset\BC$ we denote the set of all distinct eigenvalues of $L$. Throughout this article, for our arguments we assume
\begin{equation}\label{eq-cond-L3}
0\notin\si(L).
\end{equation}
For the sake of self-containedness, we collect necessary ingredients concerning the eigenvalue theory of $L$, which can be found e.g. in Kato \cite{K76}. For each $n\in\BN$, we take a circle $\ga_n$ centered at $\la_n$ with sufficiently small radius such that $\ga_n$ does not enclose $\la_{n'}$ with any $n'\ne n$. Defining an operator
\[
P_n:=-\f1{2\pi\sqrt{-1}}\int_{\ga_n}(L-z)^{-1}\,\rd z,
\]
we know that
\[
P_n P_{n'}=\begin{cases}
P_n, & n'=n,\\
0, & n'\ne n.
\end{cases}
\]
We call $P_n$ the eigenprojection for $\la_n$ and any $\vp\in P_n L^2(\Om)\setminus\{0\}$ a generalized eigenfunction of $\la_n$. Meanwhile, the space $P_n L^2(\Om)$ is finite dimensional and we set $d_n:=\dim P_n L^2(\Om)<\infty$. We note that if $L$ is symmetric, then there holds for all $n\in\BN$ that
\begin{equation}\label{eq-symm}
\la_n\in\BR,\quad P_n L^2(\Om)=\ker(L-\la_n),\quad d_n=\dim\ker(L-\la_n).
\end{equation}

Next, we introduce the eigennilpotent of $\la_n$ as
\[
D_n:=(L-\la_n)P_n.
\]
Then we can verify that $P_n\vp=\vp$ for $\vp\in P_n L^2(\Om)$ and
\[
P_n L^2(\Om)\subset\cD(L),\quad L P_n L^2(\Om)\subset P_n L^2(\Om),\quad D_n P_n L^2(\Om)\subset P_n L^2(\Om),\quad D_n^{d_n}=(L-\la_n)^{d_n}P_n=0.
\]
Using $P_n$ and $D_n$, we have the following Laurant expansion for the resolvent:
\begin{equation}\label{eq-Laurant}
(L-z)^{-1}P_n=\f{P_n}{\la_n-z}+\sum_{k=1}^{d_n-1}\f{(-1)^k D_n^k}{(\la_n-z)^{k+1}},\quad z\not\in\si(L).
\end{equation}
On the other hand, we recall the following resolvent estimate (see e.g. Tanabe \cite{T75}): there exist constants $z_0>0$ and $C>0$ such that
\begin{equation}\label{eq-tanabe}
\|(L+z)^{-1}h\|_{L^2(\Om)}\le\f C{|z|}\|h\|_{L^2(\Om)},\quad\forall\,z>z_0,\ \forall\,h\in L^2(\Om).
\end{equation}
Finally, we introduce the fractional power $L^\ga$ of $L$ along with its domain $\cD(L^\ga)$ for $\ga>0$. First for $0<\ga<1$, the fractional operator $L^{-\ga}$ is defined by
\[
L^{-\ga}h:=\f{\sin\pi\ga}\pi\int_{\BR_+}z^{-\ga}(L+z)^{-1}h\,\rd z,\quad h\in L^2(\Om).
\]
Then there holds $L^{-\ga}P_n L^2(\Om)\subset P_n L^2(\Om)$ and hence $L^\ga P_n L^2(\Om)\subset P_n L^2(\Om)$ for all $\ga\ge0$. Then for $0<\ga<1$, we define $L^\ga$ as the inverse of $L^{-\ga}$. For $\ga\in\BR_+\setminus\BN$, we define $L^\ga=L^{\ga-\lfloor\ga\rfloor}\circ L^{\lfloor\ga\rfloor}$, where $\lfloor\ga\rfloor:=\max\{n\in\mathbb Z\mid n\le\ga\}$ stands for the flooring function and $\circ$ denotes the composite. Then we set $\cD(L^\ga):=\{h\in L^2(\Om)\mid L^\ga h\in L^2(\Om)\}$ for $\ga>0$.

Now we recall the well-posedness results concerning the forward problems \eqref{eq-ibvp-u}.

\begin{lem}\label{lem-fp}
Assume {\rm\eqref{eq-cond-L1}--\eqref{eq-cond-L2}} and \eqref{eq-cond-L3}. Let $u$ satisfy \eqref{eq-ibvp-u} with $f\equiv0$ in $\Om$ and $a\in\cD(L^\ga)$ with $\ga\ge0$. Then for any fixed constant $\eta\in[0,1),$ the problem \eqref{eq-ibvp-u} admits a unique solution $u\in C([0,\infty);\cD(L^\ga))\cap C(\BR_+;\cD(L^{\ga+\eta}))$. Moreover, there exists a constant $C>0$ depending on $\al_j,q_j\ (j=1,\ldots,m),\Om,L,\ga$ such that
\begin{equation}\label{eq-est-u}
\|u(\,\cdot\,,t)\|_{\cD(L^{\ga+\eta})}\le C\,t^{-\al_1\eta}\e^{C t}\|a\|_{\cD(L^\ga)},\quad\forall\,t>0.
\end{equation}
\end{lem}

The above lemma asserts that the solution norm $\|u(\,\cdot\,,t)\|_{\cD(L^{\ga+\eta})}$ increases at most exponentially as $t\to\infty$, which guarantees the existence of its Laplace transform $\wh u(\,\cdot\,,p)$ for large $p>0$. 
Lemma \ref{lem-fp} slightly improves Li, Huang and Yamamoto \cite[Theorem 2.3]{LHY20a} in the sense of more general choice of $a$ and the key estimate \eqref{eq-est-u}. Especially, in \cite[Theorem 2.3]{LHY20a} the estimate \eqref{eq-est-u} only held for $t\in(0,T)$ with fixed $T>0$. However, a close scrutiny into the proof reveals that $T$ can be chosen arbitrarily. On the other hand, it was proved in \cite{LLY15,LHY20a} that if $\bm b\equiv\bm0$ and $c\ge0$ on $\ov\Om$\,, then the solution $u(\,\cdot\,,t)$ obeys the decay rate $t^{-\al_m}$ as $t\to\infty$. Nevertheless, instead of the sharp decay estimate, here \eqref{eq-est-u} is enough for applying the Laplace transform, which plays an essential role in the proof (see Section \ref{sec-proof}).

\section{Main Results and Examples}\label{sec-main}

Now we are well prepared to state the main results of this article. To this end, we introduce another initial-boundary value problem
\begin{equation}\label{eq-ibvp-v}
\left\{\begin{alignedat}{2}
& \!\!\left(\sum_{j=1}^{m'}r_j\pa_t^{\be_j}+L\right)v(\bm x,t)=\rho(t)g(\bm x), & \quad & (\bm x,t)\in\Om\times(0,T),\\
& v(\bm x,0)=b(\bm x), & \quad & \bm x\in\Om,\\
& v(\bm x,t)=0, & \quad & (\bm x,t)\in\pa\Om\times(0,T),
\end{alignedat}\right.
\end{equation}
where $m'\in\BN$ is a constant and $\be_j,r_j$ ($j=1,\ldots,m'$) are positive constants satisfying $1>\be_1>\be_2>\cdots>\be_{m'}>0$. We notice that \eqref{eq-ibvp-u} and \eqref{eq-ibvp-v} share the same elliptic part $L$ and the temporal component $\rho$ of the source term. Nevertheless, their initial values $a,b$ and spatial components $f,g$ of source terms differ from each other, which are basically assumed also to be unknown.

Let us state our first main result on the unique determination of orders and 
parameters with short-time 
asymptotic information for homogeneous problems, i.e., 
$f=g\equiv0$ in \eqref{eq-ibvp-u} and \eqref{eq-ibvp-v}.

\begin{thm}\label{thm-inexact}
Assume {\rm\eqref{eq-cond-L1}--\eqref{eq-cond-L2}} and \eqref{eq-cond-L3}. 
Let $u,v$ satisfy \eqref{eq-ibvp-u} and \eqref{eq-ibvp-v} respectively with $f=g\equiv0$ in $\Om,$ where $a,b\in\cD(L^\ga)$ with $\ga>1+d/4$. 
Choose $\bm x_0\in\Om$ arbitrarily such that 
$L a(\bm x_0)\ne0,L b(\bm x_0)\ne0$ and denote $\ka:=L a(\bm x_0)/L b(\bm x_0)$.
\begin{enumerate}[label=(\alph*),leftmargin=*]
\item If there exist constants $C>0$ and $\nu>\min\{\al_1,\be_1\}$ such that
\begin{equation}\label{eq-inexact}
|u(\bm x_0,t)-v(\bm x_0,t)|\le C\,t^\nu,\quad0\le t\le\tau,
\end{equation}
where $\tau>0$ is some constant$,$ then
\begin{equation}\label{eq-unique1}
\al_1=\be_1,\quad \f{q_1}{r_1}=\ka.
\end{equation}
\item If further $a,b\in\cD(L^{1+\ga})$ and \eqref{eq-inexact} is satisfied with $\nu>2\min\{\al_1,\be_1\},$ then \eqref{eq-inexact} implies
\begin{equation}\label{eq-unique2}
m=m',\quad\al_j=\be_j,\quad\f{q_j}{r_j}=\ka,\quad j=1,\ldots,m.
\end{equation}
\end{enumerate}
\end{thm}

In \eqref{eq-inexact}, we remark that $\tau>0$ can be arbitrarily small, and 
in view of \eqref{eq-asymp},
this assumption means that some principal parts of the asymptotic expansions of
$u(\bm x_0,t)$ and $v(\bm x_0,t)$ coincide.

\begin{rem}
{\rm\begin{enumerate}[label=(\alph*),leftmargin=*]
\item According to the Sobolev embedding theorem (see Adams \cite{A75}), 
we see that $a,b\in\cD(L^\ga)$ with $\ga>1+d/4$ implies $a,b\in C^2(\ov\Om)$ and thus $L a(\bm x_0),L b(\bm x_0)$ are well-defined. On the other hand, taking $t=0$ in \eqref{eq-inexact}, we automatically obtain
\begin{equation}\label{eq-a=b}
a(\bm x_0)=b(\bm x_0).
\end{equation}
\item Unlike most of the existing results on the uniqueness for inverse problems, 
we only need the short-time asymptotic behavior \eqref{eq-inexact} 
instead of their coincidence. Moreover, we even do not need $a=b$ as was assumed in most literature. Therefore, we minimize the assumptions on initial values to $L a(\bm x_0)\ne0$, $L b(\bm x_0)\ne0$ and \eqref{eq-a=b}.
\item The choices of the order $\nu$ in \eqref{eq-inexact} in Theorem \ref{thm-inexact}(i)--(ii) turn out to be the minimum necessary condition used in the proof. If there is an a priori upper bound $\ov\al$ of $\min\{\al_1,\be_1\}$, then safer requirements of $\nu$ would be $\nu>\ov\al$ in (i) and $\nu>2\ov\al$ in (ii). If there is no a priori information on $\al_1,\be_1$, then we should require $\nu\ge1$ in (i) and $\nu\ge2$ in (ii).
\item Indeed, by the arguments in \cite{LLY15}, one can prove
\[
|u(\bm x_0,t)-a(\bm x_0)|=O(t^{\al_1}),\quad|v(\bm x_0,t)-b(\bm x_0)|=O(t^{\be_1})\quad\mbox{as }t\to0,
\]
which, together with \eqref{eq-a=b}, implies
\begin{equation}\label{eq-asymp}
|u(\bm x_0,t)-v(\bm x_0,t)|\le|u(\bm x_0,t)-a(\bm x_0)|+|v(\bm x_0,t)-b(\bm x_0)|=O(t^{\min\{\al_1,\be_1\}})\quad\mbox{as }t\to0.
\end{equation}
However, \eqref{eq-inexact} does not follow automatically from \eqref{eq-asymp}.
\end{enumerate}}
\end{rem}

Since the key assumption \eqref{eq-inexact} is satisfied if
\begin{equation}\label{eq-exact}
\exists\,\tau>0\quad\mbox{such that }u(\bm x_0,t)=v(\bm x_0,t),\quad0\le t\le\tau,
\end{equation}
Theorem \ref{thm-inexact} definitely covers the uniqueness with usual exact data. On the contrary, by an argument of the reverse proposition, we can immediately obtain the following distinguishability property: if there exists a constant $\nu<\min\{\al_1,\be_1\}$ such that
\[
|u(\bm x_0,t)-v(\bm x_0,t)|\ge C\,t^\nu\quad\mbox{near }t=0,
\]
then there should be $\al_1\ne\be_1$.

In general, one can at most identify the parameters $q_j$ up to a multiplier $\ka$, but cannot further assert $q_j=r_j$ ($j=1,\ldots,m$) from \eqref{eq-inexact}. This is demonstrated by the following counterexample.

\begin{ex}\label{ex-counter}
{\rm In \eqref{eq-ibvp-u} and \eqref{eq-ibvp-v}, let us simply take
\[
\Om=(0,\pi),\quad f=g\equiv0\mbox{ in }\Om,\quad m=m'=1,\quad\al_1=\be_1,\quad q_1=4,\quad r_1=1,\quad L=-\pa_x^2.
\]
We pick any $x_0\in(0,\pi)\setminus\{\pi/2\}$ and select
\[
a(x)=\f{\sin2x}{2\cos x_0},\quad b(x)=\sin x.
\]
Then it is readily seen that
\[
a(x_0)=b(x_0)=\sin x_0,\quad\f{q_1}{r_1}=4=\f{a''(x_0)}{b''(x_0)}.
\]
However, employing the Mittag-Leffler function
\[
E_{\al_1,1}(z):=\sum_{\ell=0}^\infty\f{z^\ell}{\Ga(\al_1\ell+1)},
\]
we can easily obtain the explicit solutions
\[
u(x,t)=E_{\al_1,1}(-t^{\al_1})\f{\sin 2x}{2\cos x_0},\quad v(x,t)=E_{\al_1,1}(-t^{\al_1})\sin x
\]
and hence $u(x_0,t)=E_{\al_1,1}(-t^{\al_1})\sin x_0=v(x_0,t)$. This means that even the exact data for all $t\ge0$ fails to guarantee $q_1=r_1$.}
\end{ex}

Similarly to Theorem \ref{thm-inexact}, we can obtain the same uniqueness with 
\eqref{eq-inexact} for inhomogeneous problems, i.e., $a=b\equiv0$ in 
\eqref{eq-ibvp-u} and \eqref{eq-ibvp-v}.

\begin{thm}\label{thm-inexact1}
Assume {\rm\eqref{eq-cond-L1}--\eqref{eq-cond-L2}} and \eqref{eq-cond-L3}. Let $u,v$ satisfy \eqref{eq-ibvp-u} and \eqref{eq-ibvp-v} respectively with $a=b\equiv0$ in $\Om,$ where $f,g\in\cD(L^\ga)$ with $\ga>d/4$ and there exists a constant $\mu>-1$ such that
\begin{equation}\label{eq-rho}
\rho(t)\sim t^\mu\quad\mbox{as $t \downarrow 0$}.
\end{equation}
Fix sufficiently small $\tau>0,$ pick $\bm x_0\in\Om$ such that $f(\bm x_0)\ne0,g(\bm x_0)\ne0$ and denote $\ka:=f(\bm x_0)/g(\bm x_0)$.
\begin{enumerate}[label=(\alph*),leftmargin=*]
\item If there exist constants $C>0$ and $\nu>\min\{\al_1,\be_1\}+\mu$ such that \eqref{eq-inexact} is satisfied, then \eqref{eq-unique1} holds.
\item If further $f,g\in\cD(L^{1+\ga})$ and \eqref{eq-inexact} is satisfied with $\nu>2\min\{\al_1,\be_1\}+\mu,$ then \eqref{eq-inexact} implies \eqref{eq-unique2}.
\end{enumerate}
\end{thm}

The assumption \eqref{eq-rho} describes the short-time asymptotic behavior of the temporal component $\rho$ of the source terms in \eqref{eq-ibvp-u} and \eqref{eq-ibvp-v}, which turns out to be essential in proving Theorem \ref{thm-inexact1}. Indeed, since the information \eqref{eq-inexact} is concerned only with 
small $t>0$, we need more localized assumption on $\rho$ than 
$\rho\in L^1(\BR_+)$.
On the other hand, the solution behavior near $t=0$ is also influenced by 
$\rho$. Therefore, compared with Theorem \ref{thm-inexact}, the difference $\nu-\mu$ of orders instead of $\nu$ itself appears in the conditions guaranteeing the uniqueness in Theorem \ref{thm-inexact1}. We notice that \eqref{eq-rho} indicates $\rho\not\equiv0$ near $t=0$ and is not restrictive because it covers a wide choice of $\rho$ behaving asymptotically like a power function for small $t>0$. However, it does exclude e.g.
\[
\rho(t)=\begin{cases}
0, & t=0,\\
\exp(-1/t), & 0<t<T,
\end{cases}
\]
which satisfies $\rho^{(i)}(0)=0$ for all $i=0,1,\ldots$. 
With such functions, it is usually rather difficult to obtain 
the uniqueness or the stability for inverse problems. 
On the contrary, the ill-posedness is reduced if there exists $i=0,1,\ldots$ 
such that $\rho^{(i)}(0)\ne0$, which falls within the framework of 
\eqref{eq-rho}.

Similarly to Example \ref{ex-counter}, one can easily construct a counterexample 
against $q_j=r_j$ ($j=1,\ldots,m$) in the framework of Theorem \ref{thm-inexact1}. Owing to multiple unknown parameters, even a usual type of data 
\eqref{eq-exact} does not yield the uniqueness in determining coefficients
$q_j, r_j$.

Finally, we discuss the uniqueness and the non-uniqueness by \eqref{eq-exact}.

\begin{thm}\label{thm-unique}
Assume that the elliptic operator $L$ is symmetric$,$ i.e.$,\ \bm b\equiv\bm0$ on $\ov\Om\,$. Let
\[
\ker(L-\la_n)=\rspan\{\vp_{n,k}\mid k=1,\ldots,d_n\},
\]
where $\vp_{n,k}\ (k=1,\ldots,d_n)$ are orthonormal eigenfunctions associated with $\la_n$. Define a countable set $\Si:=\{\la_n/\la_{n'}\mid n,n'\in\BN\}\subset\BR$.
\begin{enumerate}[label=(\alph*),leftmargin=*]
\item Let $u,v$ satisfy \eqref{eq-ibvp-u} and \eqref{eq-ibvp-v} respectively with $f=g\equiv0$ in $\Om$. Under the same assumptions in Theorem $\ref{thm-inexact},$ \eqref{eq-exact} holds if and only if
\begin{equation}\label{eq-iff1}
m=m',\quad\al_j=\be_j\quad\f{q_j}{r_j}=\f{L a(\bm x_0)}{L b(\bm x_0)}=:\ka\in\Si,\quad j=1,\ldots,m
\end{equation}
and there exist nonempty sets $M_\ka,M_\ka'\subset\BN$ satisfying
\begin{equation}\label{eq-def-mm}
\{\la_n\in\si(L)\mid n\in M_\ka\}=\{\ka\la_n\mid\la_n\in\si(L),\ n\in M_\ka'\}
\end{equation}
and a bijection $\te_\ka:M_\ka\longrightarrow M_\ka'$ such that
\begin{equation}\label{eq-pab}
\begin{cases}
P_n a(\bm x_0)=P_{\te_\ka(n)} b(\bm x_0),\ P_n L a(\bm x_0)=\ka\,P_{\te_\ka(n)} L b(\bm x_0), & n\in M_\ka,\\
P_n a(\bm x_0)=P_n L a(\bm x_0)=0, & n\not\in M_\ka,\\
P_n b(\bm x_0)=P_n L b(\bm x_0)=0, & n\not\in M_\ka'.
\end{cases}
\end{equation}
\item Let $u,v$ satisfy \eqref{eq-ibvp-u} and \eqref{eq-ibvp-v} respectively with $a=b\equiv0$ in $\Om$. Under the same assumptions in Theorem $\ref{thm-inexact1},$ \eqref{eq-exact} holds if and only if
\begin{equation}\label{eq-iff2}
m=m',\quad\al_j=\be_j,\quad\f{q_j}{r_j}=\f{f(\bm x_0)}{g(\bm x_0)}=:\ka\in\Si,\quad j=1,\ldots,m
\end{equation}
and there exist the same sets $M_\ka,M_\ka'\subset\BN$ and bijection $\te_\ka:M_\ka\longrightarrow M_\ka'$ as that in {\rm(1)} such that
\begin{equation}\label{eq-pfg}
\begin{cases}
P_n f(\bm x_0)=\ka\,P_{\te_\ka(n)} g(\bm x_0), & n\in M_\ka,\\
P_n f(\bm x_0)=0, & n\not\in M_\ka,\\
P_n g(\bm x_0)=0, & n\not\in M_\ka'.
\end{cases}
\end{equation}
\end{enumerate}
\end{thm}

\begin{rem}\label{rem-ka1}
{\rm Theorem \ref{thm-unique} characterizes the coincidence \eqref{eq-exact} of observation data via initial values $a,b$ or source terms $f,g$. From the definition of $\Si$, it is obvious that $1\in\Si$. We discuss the cases of $\ka=1$, $\ka\notin\Si$ and $\ka\in\Si\setminus\{1\}$ separately.
\begin{enumerate}[label=(\alph*),leftmargin=*]
\item If $\ka=1$, i.e., $L a(\bm x_0)=L b(\bm x_0)$ or $f(\bm x_0)=g(\bm x_0)$, then $M_1=M_1'=\BN$ and $\te_1$ is the identity operator. 
In this case, it is readily seen that \eqref{eq-pab} and \eqref{eq-pfg} are equivalent to
\begin{gather}
P_n a(\bm x_0)=P_n b(\bm x_0),\quad P_n L a(\bm x_0)=P_n L b(\bm x_0),\quad\forall\,n\in\BN,\label{eq-pab1}\\
P_n f(\bm x_0)=P_n g(\bm x_0),\quad\forall\,n\in\BN,\nonumber
\end{gather}
respectively.
\item If $\ka\notin\Si$, then Theorem \ref{thm-unique} asserts that the observation data $u(\bm x_0,t)$ and $v(\bm x_0,t)$ ($0\le t\le\tau$) cannot coincide.
\item If $\ka\in\Si\setminus\{1\}$, then it is indicated in \eqref{eq-pab} and \eqref{eq-pfg} that the initial values $a,b$ or the source terms $f,g$ should satisfy a rather special relation. To illustrate this point, we provide the following example.
\end{enumerate}}
\end{rem}

\begin{ex}\label{ex-special}
{\rm Similarly to Example \ref{ex-counter}, in \eqref{eq-ibvp-u} and \eqref{eq-ibvp-v} we take
\[
\Om=(0,\pi),\quad f=g\equiv0\mbox{ in }\Om,\quad \ka=\f{L a(x_0)}{L b(x_0)}=4,\quad L=-\pa_x^2.
\]
Then it is readily seen that
\[
\la_n=n^2,\quad\Si=\{(n/n')^2\mid n,n'\in\BN\}=\mathbb Q^2\setminus\{0\},\quad P_n f=(f,\vp_n)\vp_n,\ f\in L^2(0,\pi),
\]
where $\vp_n(x):=\sqrt{2/\pi}\sin n x$ ($n\in\BN$). Further, by the definitions of $M_\ka,M'_\ka$ and $\te_\ka$, it is straightforward to verify that $M_4=2\BN,M'_4=\BN$ and $\te_4(n)=n/2$. Therefore, the relation \eqref{eq-pab} can be rephrased as
\[
(a,\vp_{2n})\vp_{2n}(x_0)=(b,\vp_n)\vp(x_0),\quad(a,\vp_{2n-1})\vp_{2n-1}(x_0)=0,\quad n\in\BN.
\]
Especially, if $x_0\ne\pi\mathbb Q$, i.e., $x_0$ is not a zero of $\vp_n$ ($\forall\,n\in\BN$), then $a$ should take the special form of
\[
a(x)=\sum_{n=1}^\infty\f{(b,\vp_n)}{\vp_{2n}(x_0)}\vp_{2n}(x_0)=\f1{\sqrt{2\pi}}\sum_{n=1}^\infty\f{(b,\vp_n)}{\cos n x_0}\sin2n x,
\]
which is a generalization of Example \ref{ex-counter}.}
\end{ex}

On the same direction of Theorem \ref{thm-unique} and Example \ref{ex-special}, it turns out that in a special case, we can even obtain the uniqueness of 
$a,b$ or $f,g$.

\begin{coro}\label{coro-unique}
In addition to the same assumptions in Theorem $\ref{thm-unique},$ 
we assume that all the eigenvalues of $L$ are simple and
\begin{equation}\label{eq-omega}
\bm x_0\not\in\om:=\bigcup_{n=1}^\infty\{\bm x\in\Om\mid\vp_n(\bm x)=0\}\subset\Om,
\end{equation}
where $\vp_n$ is an eigenfunction associated with $\la_n$. 
Let $\Si,M_\ka,M_\ka',\te_\ka$ be the same as that in Theorem 
$\ref{thm-unique}$.
\begin{enumerate}[label=(\alph*),leftmargin=*]
\item Let $u,v$ satisfy \eqref{eq-ibvp-u} and \eqref{eq-ibvp-v} respectively with $f=g\equiv0$ in $\Om$. Then \eqref{eq-exact} holds if and only if \eqref{eq-iff1} holds and
\begin{equation}\label{eq-rel-ab}
\begin{aligned}
& a=\sum_{n\in M_\ka}(a,\vp_n)\vp_n,\quad b=\sum_{n\in M_\ka'}(b,\vp_n)\vp_n\\
& \mbox{with }(a,\vp_n)\vp_n(\bm x_0)=(b,\vp_{\te_\ka(n)})\vp_{\te_\ka(n)}(\bm x_0),\ n\in M_\ka.
\end{aligned}
\end{equation}
Especially if $\ka=1,$ i.e. $L a(\bm x_0)=L b(\bm x_0),$ then \eqref{eq-rel-ab} is equivalent to $a\equiv b$ in $\Om$.
\item Let $u,v$ satisfy \eqref{eq-ibvp-u} and \eqref{eq-ibvp-v} respectively with $a=b\equiv0$ in $\Om$. Then \eqref{eq-exact} holds if and only if \eqref{eq-iff2} holds and
\begin{equation}\label{eq-rel-fg}
\begin{aligned}
& f=\sum_{n\in M_\ka}(f,\vp_n)\vp_n,\quad g=\sum_{n\in M_\ka'}(g,\vp_n)\vp_n\\
& \mbox{with }(f,\vp_n)\vp_n(\bm x_0)=\ka(g,\vp_{\te_\ka(n)})\vp_{\te_\ka(n)}(\bm x_0),\ n\in M_\ka.
\end{aligned}
\end{equation}
Especially if $\ka=1,$ i.e. $f(\bm x_0)=g(\bm x_0),$ then \eqref{eq-rel-fg} is equivalent to $f\equiv g$ in $\Om$.
\end{enumerate}
\end{coro}

The assumption \eqref{eq-omega} is a rank 
condition, which is widely used for observability of partial 
differential equations (e.g. \cite{S75}). In some special cases such as that in Example \ref{ex-special}, the Lebesgue measure of $\om$ defined in \eqref{eq-omega} is zero. Therefore, the uniqueness claimed in Corollary \ref{coro-unique} is generic, which is rather surprising because the single point observation can even almost determine the initial value or the spatial component of the source term uniquely.

\section{Proofs of Main Results}\label{sec-proof}

For the proof of Theorem \ref{thm-inexact}, we need the completeness of all the generalized eigenfunctions of $L$.

\begin{lem}\label{lem-total}
\begin{enumerate}[label=(\alph*),leftmargin=*]
\item For any $h\in L^2(\Om),$ there exists a sequence $\{h_N\}_{N\in\BN}\subset L^2(\Om)$ such that
\[
h_N\in\sum_{n=1}^N P_n L^2(\Om),\quad\lim_{N\to\infty}\|h-h_N\|_{L^2(\Om)}=0.
\]
\item For any $h\in\cD(L^\ga)$ with $\ga\ge0,$ there exists a sequence $\{h_N\}_{N\in\BN}\subset\cD(L^\ga)$ such that
\[
h_N\in\sum_{n=1}^N P_n L^2(\Om),\quad\lim_{N\to\infty}\|L^\ga(h-h_N)\|_{L^2(\Om)}=0.
\]
\end{enumerate}
\end{lem}

\begin{proof}
\begin{enumerate}[label=(\alph*),leftmargin=*]
\item It follows immediately from the density of the linear subspace spanned by all generalized eigenfunctions of $L$ in $L^2(\Om)$ (see Agmon \cite{A65}).
\item Since $h\in\cD(L^\ga)$ implies $L^\ga h\in L^2(\Om)$, we can apply (i) to choose a sequence $\{\psi_N\}_{N\in\BN}\subset L^2(\Om)$ such that
\[
\psi_N\in\sum_{n=1}^N P_n L^2(\Om),\quad\lim_{N\to\infty}\|\psi_N-L^\ga h\|_{L^2(\Om)}=0.
\]
Owing to \eqref{eq-cond-L3}, we see that $L^{-\ga}:L^2(\Om)\longrightarrow L^2(\Om)$ exists and is bounded, indicating
\[
\lim_{N\to\infty}\|L^\ga(L^{-\ga}\psi_N)-L^\ga h\|_{L^2(\Om)}=0.
\]
By the definition of $L^{-\ga}$, we see that $h_N:=L^{-\ga}\psi_N\in\sum_{n=1}^N P_n L^2(\Om)\subset\cD(L^\ga)$ and thus
\[
\lim_{N\to\infty}\|L^\ga h_N-L^\ga h\|_{L^2(\Om)}=0.
\]
The proof of Lemma \ref{lem-total} is complete.
\end{enumerate}
\end{proof}

\begin{proof}[Proof of Theorem $\ref{thm-inexact}$]
Let $f=g\equiv0$ in $\Om$ in \eqref{eq-ibvp-u} and \eqref{eq-ibvp-v}. We divide the proof into 4 steps.

{\bf Step 1 } To begin with, we make some overall preparations. Since the assumption $\ga-1>d/4$ implies the Sobolev embedding $\cD(L^{\ga-1})\subset C(\ov\Om)$, it follows from Lemma \ref{lem-fp} that $u(\bm x_0,t)$ is well-defined for $t>0$ and fulfills the estimate
\[
|u(\bm x_0,t)|\le\|u(\,\cdot\,,t)\|_{C(\ov\Om)}\le C\|u(\,\cdot\,,t)\|_{\cD(L^{\ga-1})}\le C\,\e^{C t}\|a\|_{\cD(L^{\ga-1})},\quad t>0.
\]
The similar estimate also holds for $v(\bm x_0,t)$, which, together with the key assumption \eqref{eq-inexact}, indicates
\[
|u(\bm x_0,t)-v(\bm x_0,t)|\le\begin{cases}
C\,t^\nu, & 0\le t\le\tau,\\
C\,\e^{C t}, & t>\tau.
\end{cases}
\]
Then for large $p>0$, we estimate the Laplace transform of $u(\bm x_0,t)-v(\bm x_0,t)$ from above as
\begin{align*}
\wh u(\bm x_0,p)-\wh v(\bm x_0,p) & \le\int_{\BR_+}\e^{-p t}|u(\bm x_0,t)-v(\bm x_0,t)|\,\rd t\le C\int_0^\tau\e^{-p t}t^\nu\,\rd t+C\int_\tau^\infty\e^{-(p-C)t}\,\rd t\\
& \le C\int_{\BR_+}\e^{-p t}t^\nu\,\rd t+\f{C\,\e^{-(p-C)\tau}}{p-C}=\f{C\,\Ga(\nu+1)}{p^{\nu+1}}+o(p^{-\nu-1})\le C\,p^{-\nu-1}.
\end{align*}
We can similarly obtain the lower estimate $\wh u(\bm x_0,p)-\wh v(\bm x_0,p)\ge-C\,p^{-\nu-1}$ for large $p>0$ and hence
\begin{equation}\label{eq-est0}
\wh u(\bm x_0,p)-\wh v(\bm x_0,p)=O(p^{-\nu-1})\quad\mbox{for large }p>0.
\end{equation}

Now we shall represent $\wh u(\bm x_0,p)$ and $\wh v(\bm x_0,p)$ explicitly for
large $p > 0$. Performing the Laplace transform in \eqref{eq-ibvp-u} and using the formula
\[
\wh{\pa_t^\al h}(p)=p^\al\wh h(p)-p^{\al-1}h(0),
\]
we derive the boundary value problem for an elliptic equation for $\wh u(\,\cdot\,,p)$:
\[
\left\{\begin{alignedat}{2}
& \!\!\left(\sum_{j=1}^m q_j p^{\al_j}+L\right)\wh u(\,\cdot\,,p)=a\sum_{j=1}^m q_j p^{\al_j-1} & \quad & \mbox{in }\Om,\\
& \wh u(\,\cdot\,,p)=0 & \quad & \mbox{on }\pa\Om.
\end{alignedat}
\right.
\]
Transforming \eqref{eq-ibvp-v} in the same manner and taking the inverses, we obtain
\begin{align}
\wh u(\,\cdot\,,p) & =\f1p\sum_{j=1}^m q_j p^{\al_j}\left(L+\sum_{j=1}^m q_j p^{\al_j}\right)^{-1}a,\label{eq-exp-u}\\
\wh v(\,\cdot\,,p) & =\f1p\sum_{j=1}^{m'}r_j p^{\be_j}\left(L+\sum_{j=1}^{m'}r_j p^{\be_j}\right)^{-1}b.\label{eq-exp-v}
\end{align}
According to Lemma \ref{lem-total}, there exist $a_N=\sum_{n=1}^N a_{N,n}$ and $b_N=\sum_{n=1}^N b_{N,n}$ ($N\in\BN$) such that
\begin{equation}\label{eq-cov1}
a_{N,n},b_{N,n}\in P_n L^2(\Om),\quad\lim_{N\to\infty}\|L^\ga(a-a_N)\|_{L^2(\Om)}=\lim_{N\to\infty}\|L^\ga(b-b_N)\|_{L^2(\Om)}=0.
\end{equation}
Especially, since $\ga>1+d/4$, the Sobolev embedding implies
\begin{equation}\label{eq-cov2}
\lim_{N\to\infty}L a_N(\bm x_0)=L a(\bm x_0),\quad\lim_{N\to\infty}L b_N(\bm x_0)=L b(\bm x_0).
\end{equation}
Then for any fixed $N\in\BN$, we decompose $a=a_N+(a-a_N)$ and employ the expansion \eqref{eq-Laurant} to rewrite \eqref{eq-exp-u} as
\begin{align}
p\,\wh u(\,\cdot\,,p) & =\sum_{j=1}^m q_j p^{\al_j}\left(L+\sum_{j=1}^m q_j p^{\al_j}\right)^{-1}\left\{\sum_{n=1}^N a_{N,n}+(a-a_N)\right\}\nonumber\\
& =\sum_{j=1}^m q_j p^{\al_j}\sum_{n=1}^N\left(L+\sum_{j=1}^m q_j p^{\al_j}\right)^{-1}P_n a_{N,n}\nonumber\\
& \quad\,+\left\{\left(L+\sum_{j=1}^m q_j p^{\al_j}\right)-L\right\}\left(L+\sum_{j=1}^m q_j p^{\al_j}\right)^{-1}(a-a_N)\nonumber\\
& =\sum_{j=1}^m q_j p^{\al_j}\sum_{n=1}^N\sum_{k=0}^{d_n-1}\f{(-1)^k D_n^k a_{N,n}}{(\la_n+\sum_{j=1}^m q_j p^{\al_j})^{k+1}}+(a-a_N)-L\left(L+\sum_{j=1}^m q_j p^{\al_j}\right)^{-1}(a-a_N)\nonumber\\
& =\sum_{n=1}^N\left(1-\f{\la_n}{\la_n+\sum_{j=1}^m q_j p^{\al_j}}\right)a_{N,n}+\sum_{j=1}^m q_j p^{\al_j}\sum_{n=1}^N\sum_{k=1}^{d_n-1}\f{(-1)^k D_n^k a_{N,n}}{(\la_n+\sum_{j=1}^m q_j p^{\al_j})^{k+1}}\nonumber\\
& \quad\,+a-a_N-\left(L+\sum_{j=1}^m q_j p^{\al_j}\right)^{-1}L(a-a_N)\nonumber\\
& =a-\sum_{n=1}^N\f{\la_n a_{N,n}}{\la_n+\sum_{j=1}^m q_j p^{\al_j}}+\sum_{j=1}^m q_j p^{\al_j}\sum_{n=1}^N\sum_{k=1}^{d_n-1}\f{(-1)^k D_n^k a_{N,n}}{(\la_n+\sum_{j=1}^m q_j p^{\al_j})^{k+1}}\nonumber\\
& \quad\,-\left(L+\sum_{j=1}^m q_j p^{\al_j}\right)^{-1}L(a-a_N).\label{eq-exp-u1}
\end{align}
Rewriting \eqref{eq-exp-v} in the same manner, we obtain
\begin{align}
p\,\wh v(\,\cdot\,,p) & =b-\sum_{n=1}^N\f{\la_n b_{N,n}}{\la_n+\sum_{j=1}^{m'}r_j p^{\be_j}}+\sum_{j=1}^{m'}r_j p^{\be_j}\sum_{n=1}^N\sum_{k=1}^{d_n-1}\f{(-1)^k D_n^k b_{N,n}}{(\la_n+\sum_{j=1}^{m'}r_j p^{\be_j})^{k+1}}\nonumber\\
& \quad\,-\left(L+\sum_{j=1}^{m'}r_j p^{\be_j}\right)^{-1}L(b-b_N).\label{eq-exp-v1}
\end{align}
Taking $\bm x=\bm x_0$ in \eqref{eq-exp-u1}--\eqref{eq-exp-v1} and substituting them into \eqref{eq-est0}, we notice \eqref{eq-a=b} to arrive at
\begin{equation}\label{eq-est1}
\sum_{i=1}^3(I_N^i(p)-J_N^i(p))=O(p^{-\nu})\quad\mbox{for large } p>0,
\end{equation}
where
\begin{alignat*}{2}
I_N^1(p) & :=\sum_{n=1}^N\f{\la_n a_{N,n}(\bm x_0)}{\la_n+\sum_{j=1}^m q_j p^{\al_j}}, & \quad J_N^1(p) & :=\sum_{n=1}^N\f{\la_n b_{N,n}(\bm x_0)}{\la_n+\sum_{j=1}^{m'}r_j p^{\be_j}},\\
I_N^2(p) & :=\sum_{j=1}^m q_j p^{\al_j}\sum_{n=1}^N\sum_{k=1}^{d_n-1}\f{(-1)^{k+1}(D_n^k a_{N,n})(\bm x_0)}{(\la_n+\sum_{j=1}^m q_j p^{\al_j})^{k+1}}, & \quad J_N^2(p) & :=\sum_{j=1}^{m'}r_j p^{\be_j}\sum_{n=1}^N\sum_{k=1}^{d_n-1}\f{(-1)^{k+1}(D_n^k b_{N,n})(\bm x_0)}{(\la_n+\sum_{j=1}^{m'}r_j p^{\be_j})^{k+1}},\\
I_N^3(p) & :=\left(\left(L+\sum_{j=1}^m q_j p^{\al_j}\right)^{-1}L(a-a_N)\right)(\bm x_0), & \quad J_N^3(p) & :=\left(\left(L+\sum_{j=1}^{m'}r_j p^{\be_j}\right)^{-1}L(b-b_N)\right)(\bm x_0).
\end{alignat*}
Here for $I_N^3(p)$, it follows from the Sobolev embedding and the estimate \eqref{eq-tanabe} that
\begin{align}
|I_N^3(p)| & \le\left\|\left(L+\sum_{j=1}^m q_j p^{\al_j}\right)^{-1}L(a-a_N)\right\|_{C(\ov\Om)}\le C\left\|\left(L+\sum_{j=1}^m q_j p^{\al_j}\right)^{-1}L(a-a_N)\right\|_{\cD(L^{\ga-1})}\nonumber\\
& \le C\left\|\left(L+\sum_{j=1}^m q_j p^{\al_j}\right)^{-1}L^\ga(a-a_N)\right\|_{L^2(\Om)}\le\f C{\sum_{j=1}^m q_j p^{\al_j}}\|L^\ga(a-a_N)\|_{L^2(\Om)}\nonumber\\
& \le C\,p^{-\al_1}\|a-a_N\|_{\cD(L^\ga)}.\label{eq-est-i3}
\end{align}
Similarly, for $J_N^3(p)$ we have
\begin{equation}\label{eq-est-j3}
|J_N^3(p)|\le C\,p^{-\be_1}\|b-b_N\|_{\cD(L^\ga)}.
\end{equation}
Finally, we note that the regularity of $a,b$ guarantees the uniform boundedness of
\[
\sum_{n=1}^N\la_n a_{N,n}(\bm x_0),\quad\sum_{n=1}^N\sum_{k=1}^{d_n-1}(-1)^{k+1}(D_n^k a_{N,n})(\bm x_0),\quad\sum_{n=1}^N\la_n b_{N,n}(\bm x_0),\quad\sum_{n=1}^N\sum_{k=1}^{d_n-1}(-1)^{k+1}(D_n^k b_{N,n})(\bm x_0)
\]
for all $N\in\BN$. Especially, by the identity $L P_n=\la_n P_n+D_n$ we know that
\begin{equation}\label{eq-lan}
\begin{aligned}
\sum_{n=1}^N\la_n a_{N,n}(\bm x_0)+\sum_{n=1}^N D_n a_{N,n}(\bm x_0) & =\sum_{n=1}^N L a_{N,n}(\bm x_0)=L a_N(\bm x_0),\\
\sum_{n=1}^N\la_n b_{N,n}(\bm x_0)+\sum_{n=1}^N D_n b_{N,n}(\bm x_0) & =\sum_{n=1}^N L b_{N,n}(\bm x_0)=L b_N(\bm x_0).
\end{aligned}
\end{equation}

{\bf Step 2 } Now we are well prepared to prove part (i). We starting with showing $\al_1=\be_1$ by contradiction.

First let us assume that $\al_1<\be_1$. Our strategy is first multiplying both sides of \eqref{eq-est1} by $p^{\al_1}$, then passing $N\to\infty$ and finally passing $p\to\infty$ to deduce a contradiction. To this end, we calculate $p^{\al_1}I_N^1(p)$, $p^{\al_1}J_N^1(p)$, $p^{\al_1}I_N^2(p)$ and $p^{\al_1}J_N^2(p)$ as
\begin{align}
p^{\al_1}I_N^1(p) & =\sum_{n=1}^N\f{\la_n a_{N,n}(\bm x_0)}{\la_n p^{-\al_1}+\sum_{j=1}^m q_j p^{\al_j-\al_1}}=\f{1+o(1)}{q_1}\sum_{n=1}^N\la_n a_{N,n}(\bm x_0),\label{eq-est-i1}\\
p^{\al_1}J_N^1(p) & =\sum_{n=1}^N\f{\la_n b_{N,n}(\bm x_0)}{\la_n p^{-\al_1}+\sum_{j=1}^{m'}r_j p^{\be_j-\al_1}}=\f{p^{\al_1-\be_1}(1+o(1))}{r_1}\sum_{n=1}^N\la_n b_{N,n}(\bm x_0)=O(p^{\al_1-\be_1}),\label{eq-est-j1}\\
p^{\al_1}I_N^2(p) & =\sum_{j=1}^m q_j p^{\al_j-\al_1}\sum_{n=1}^N\left\{\f{D_n a_{N,n}(\bm x_0)}{(\la_n p^{-\al_1}+\sum_{j=1}^m q_j p^{\al_j-\al_1})^2}+O(p^{-\al_1})\right\}\nonumber\\
& =\f{1+o(1)}{q_1}\sum_{n=1}^N D_n a_{N,n}(\bm x_0)+O(p^{-\al_1}),\label{eq-est-i2}\\
p^{\al_1}J_N^2(p) & =O(p^{\al_1+\be_1})\sum_{n=1}^N O(p^{-2\be_1})\sum_{k=1}^{d_n-1}(-1)^{k+1}(D_n^k b_{N,n})(\bm x_0)=O(p^{\al_1-\be_1})\label{eq-est-j2}
\end{align}
for large $p>0$. Substituting \eqref{eq-est-i1}--\eqref{eq-est-j2} and \eqref{eq-est-i3}--\eqref{eq-est-j3} into \eqref{eq-est1}, we obtain
\begin{align*}
O(p^{-\nu+\al_1}) & =p^{\al_1}\sum_{i=1}^3(I_N^i(p)-J_N^i(p))\\
& =\f{1+o(1)}{q_1}\left\{\sum_{n=1}^N\la_n a_{N,n}(\bm x_0)+\sum_{n=1}^N D_n a_{N,n}(\bm x_0)\right\}+O(p^{\al_1-\be_1})+O(p^{-\al_1})\\
& \quad\,+O(1)\|a-a_N\|_{\cD(L^\ga)}+O(p^{\al_1-\be_1})\|b-b_N\|_{\cD(L^\ga)}\\
& =\f{1+o(1)}{q_1}L a_N(\bm x_0)+O(1)\|a-a_N\|_{\cD(L^\ga)}+o(1)\quad
\mbox{for large }p>0,
\end{align*}
where we used \eqref{eq-lan} and the uniform boundedness of $\|b-b_N\|_{\cD(L^\ga)}$ with respect to $N$. By the assumption $\nu>\min\{\al_1,\be_1\}=\al_1$, the above equation becomes
\begin{equation}\label{eq-est2}
\f{1+o(1)}{q_1}L a_N(\bm x_0)+O(1)\|a-a_N\|_{\cD(L^\ga)}=o(1)\quad\mbox{for 
large }p>0.
\end{equation}
Then we pass $N\to\infty$ in \eqref{eq-est2} and employ \eqref{eq-cov1}--\eqref{eq-cov2} to deduce
\[
\f{1+o(1)}{q_1}L a(\bm x_0)=o(1)\quad\mbox{for large }p>0.
\]
Consequently, passing $p\to\infty$ yields $L a(\bm x_0)=0$, which contradicts with the assumption $L a(\bm x_0)\ne0$.

In a completely parallel manner, assuming $\al_1>\be_1$ results in a contradiction with the assumption $L b(\bm x_0)\ne0$, indicating the only possibility $\al_1=\be_1$.

Now it suffices to verify $q_1/r_1=L a(\bm x_0)/L b(\bm x_0)$. Similarly to the treatment for \eqref{eq-est-i1} and \eqref{eq-est-i2}, we substitute $\al_1=\be_1$ into $p^{\al_1}J_N^1(p)$ and $p^{\al_1}J_N^2(p)$ to see that now \eqref{eq-est-j1} and \eqref{eq-est-j2} become
\[
p^{\al_1}J_N^1(p)=\f{1+o(1)}{r_1}\sum_{n=1}^N\la_n b_{N,n}(\bm x_0),\quad p^{\al_1}J_N^2(p)=\f{1+o(1)}{r_1}\sum_{n=1}^N D_n b_{N,n}(\bm x_0)+O(p^{-\al_1}),
\]
respectively. Substituting the above equalities, \eqref{eq-est-i1}, \eqref{eq-est-i2}, \eqref{eq-est-i3}, \eqref{eq-est-j3} into \eqref{eq-est1} and repeating the same calculation as before, we arrive at
\[
(1+o(1))\left(\f{L a_N(\bm x_0)}{q_1}-\f{L b_N(\bm x_0)}{r_1}\right)+O(1)\left(\|a-a_N\|_{\cD(L^\ga)}+\|b-b_N\|_{\cD(L^\ga)}\right)=o(1)
\]
for large $p>0$. Again passing first $N\to\infty$ and then $p\to\infty$, we conclude
\[
\f{L a(\bm x_0)}{q_1}-\f{L b(\bm x_0)}{r_1}=0
\]
or equivalently $q_1/r_1=L a(\bm x_0)/L b(\bm x_0)$. This completes the proof of part (i).

{\bf Step 3 } From now on we proceed to the proof of part (ii). For later convenience, we set $\ka:=L a(\bm x_0)/L b(\bm x_0)$. Notice that owing to the further assumption $a,b\in\cD(L^{1+\ga})$, the sequences
\[
\sum_{n=1}^N\la_n^2a_{N,n}(\bm x_0),\quad\sum_{n=1}^N\la_n^2b_{N,n}(\bm x_0),\quad\sum_{n=1}^N\la_n(D_n a_{N,n})(\bm x_0),\quad\sum_{n=1}^N\la_n(D_n b_{N,n})(\bm x_0)
\]
are also uniformly bounded for all $N\in\BN$.

In this step, we shall show by induction that for $j=1,2,\ldots,\min\{m,m'\}$, there holds
\begin{equation}\label{eq-induct}
\al_j=\be_j,\quad \f{q_j}{r_j}=\ka.
\end{equation}
Since the assumption of $\nu$ in part (ii) is stronger than that in part (i), the case of $j=1$ was already shown in Step 2.

Now let us assume that for some $\ell=2,\ldots,\min\{m,m'\}$, \eqref{eq-induct} holds for $j=1,\ldots,\ell-1$. Then our aim is to show that \eqref{eq-induct} also holds for $j=\ell$. Notice that now the terms $J_N^1(p)$ and $J_N^2(p)$ 
are written as
\begin{equation}\label{eq-j12}
\begin{aligned}
J_N^1(p) & =\sum_{n=1}^N\f{\la_n b_{N,n}(\bm x_0)}{\la_n+\ka^{-1}\sum_{j=1}^{\ell-1}q_j p^{\al_j}+\sum_{j=\ell}^{m'}r_j p^{\be_j}},\\
J_N^2(p) & =\left(\ka^{-1}\sum_{j=1}^{\ell-1}q_j p^{\al_j}+\sum_{j=\ell}^{m'}r_j p^{\be_j}\right)\sum_{n=1}^N\sum_{k=1}^{d_n-1}\f{(-1)^{k+1}(D_n^k b_{N,n})(\bm x_0)}{(\la_n+\ka^{-1}\sum_{j=1}^{\ell-1}q_j p^{\al_j}+\sum_{j=\ell}^{m'}r_j p^{\be_j})^{k+1}}.
\end{aligned}
\end{equation}
Similarly to Step 2, we start with showing $\al_\ell=\be_\ell$ by contradiction.

Again let us first assume that $\al_\ell<\be_\ell$. Now our strategy is first multiplying both sides of \eqref{eq-est1} by $p^{2\al_1-\be_\ell}$, then passing $N\to\infty$ and finally passing $p\to\infty$ to deduce a contradiction.

We first calculate $p^{2\al_1-\be_\ell}(I_N^1(p)-J_N^1(p))$. To this end, we rewrite the denominator in $I_N^1(p)$ as
\begin{align}
\f1{\la_n+\sum_{j=1}^m q_j p^{\al_j}} & =\f1{\la_n+\sum_{j=1}^{\ell-1}q_j p^{\al_j}}\left(1+\f{\sum_{j=\ell}^m q_j p^{\al_j}}{\la_n+\sum_{j=1}^{\ell-1}q_j p^{\al_j}}\right)^{-1}\nonumber\\
& =\f1{\la_n+\sum_{j=1}^{\ell-1}q_j p^{\al_j}}-\f{\sum_{j=\ell}^m q_j p^{\al_j}}{(\la_n+\sum_{j=1}^{\ell-1}q_j p^{\al_j})^2}+O\left(\f{(\sum_{j=\ell}^m q_j p^{\al_j})^2}{(\la_n+\sum_{j=1}^{\ell-1}q_j p^{\al_j})^3}\right)\nonumber\\
& =\f{p^{-\al_1}}{\la_n p^{-\al_1}+\sum_{j=1}^{\ell-1}q_j p^{\al_j-\al_1}}-\f{p^{\al_\ell-2\al_1}(q_\ell+o(1))}{(q_1+o(1))^2}+O(p^{2\al_\ell-3\al_1})\quad\mbox{for large }p>0.\label{eq-i1-denomi1}
\end{align}
For the denominator of the first term on the right-hand side, we further expand
\begin{align*}
\f1{\la_n p^{-\al_1}+\sum_{j=1}^{\ell-1}q_j p^{\al_j-\al_1}} & =\f1{\sum_{j=1}^{\ell-1}q_j p^{\al_j-\al_1}}\left(1+\f{\la_n p^{-\al_1}}{\sum_{j=1}^{\ell-1}q_j p^{\al_j-\al_1}}\right)^{-1}\\
& =\f1{\sum_{j=1}^{\ell-1}q_j p^{\al_j-\al_1}}+O\left(\f{\la_n p^{-\al_1}}{(\sum_{j=1}^{\ell-1}q_j p^{\al_j-\al_1})^2}\right)\\
& =\f1{\sum_{j=1}^{\ell-1}q_j p^{\al_j-\al_1}}+\la_n O(p^{-\al_1})
\quad\mbox{for large }p>0.
\end{align*}
Plugging the above expansion into \eqref{eq-i1-denomi1}, we obtain
\begin{equation}\label{eq-i1-denomi2}
\f1{\la_n+\sum_{j=1}^m q_j p^{\al_j}}=\f{p^{-\al_1}}{\sum_{j=1}^{\ell-1}q_j p^{\al_j-\al_1}}-\f{p^{\al_\ell-2\al_1}(q_\ell+o(1))}{(q_1+o(1))^2}+\la_n O(p^{-2\al_1})+O(p^{2\al_\ell-3\al_1})
\end{equation}
for large $p>0$. Treating the denominator in $J_N^1(p)$ similarly and substituting them into $p^{2\al_1-\be_\ell}(I_N^1(p)-J_N^1(p))$, we deduce
\begin{align}
p^{2\al_1-\be_\ell}(I_N^1(p)-J_N^1(p)) & =p^{\al_1-\be_\ell}\sum_{n=1}^N\la_n\left(\f{a_{N,n}(\bm x_0)}{\sum_{j=1}^{\ell-1}q_j p^{\al_j-\al_1}}-\f{b_{N,n}(\bm x_0)}{\ka^{-1}\sum_{j=1}^{\ell-1}q_j p^{\al_j-\al_1}}\right)\nonumber\\
& \quad\,-p^{\al_\ell-\be_\ell}(q_\ell+o(1))\sum_{n=1}^N\f{\la_n a_{N,n}(\bm x_0)}{(q_1+o(1))^2}+(r_\ell+o(1))\sum_{n=1}^N\f{\la_n b_{N,n}(\bm x_0)}{(r_1+o(1))^2}\nonumber\\
& \quad\,+O(p^{-\be_\ell})\sum_{n=1}^N\la_n^2(a_{N,n}-b_{N,n})(\bm x_0)+O(p^{2\al_\ell-\al_1-\be_\ell})\sum_{n=1}^N\la_n a_{N,n}(\bm x_0)\nonumber\\
& \quad\,+O(p^{\be_\ell-\al_1})\sum_{n=1}^N\la_n b_{N,n}(\bm x_0)\nonumber\\
& =\f{p^{\al_1-\be_\ell}(1+o(1))}{q_1}\sum_{n=1}^N\la_n(a_{N,n}-\ka\,b_{N,n})(\bm x_0)+\f{r_\ell+o(1)}{r_1^2}\sum_{n=1}^N\la_n b_{N,n}(\bm x_0)\nonumber\\
& \quad\,+o(1)\quad\mbox{for large }p>0,          \label{eq-ij1}
\end{align}
where we utilized the assumption $\al_1<\be_1$ and the uniform boundedness of all the involved sequences with respect to $N$.

Next we calculate $p^{2\al_1-\be_\ell}(I_N^2(p)-J_N^2(p))$. We further decompose $I_N^2(p)$ as
\begin{align*}
I_N^2(p)  & =\sum_{j=1}^m q_j p^{\al_j}\sum_{n=1}^N\left\{\f{D_n a_{N,n}(\bm x_0)}{(\la_n+\sum_{j=1}^m q_j p^{\al_j})^2}+\sum_{k=2}^{d_n-1}\f{(-1)^{k+1}(D_n a_{N,n})(\bm x_0)}{(\la_n+\sum_{j=1}^m q_j p^{\al_j})^{k+1}}\right\}\\
& =\sum_{n=1}^N\left\{\sum_{j=1}^m q_j p^{\al_j}\f{D_n a_{N,n}(\bm x_0)}{(\la_n+\sum_{j=1}^m q_j p^{\al_j})^2}+O(p^{-2\al_1})\sum_{k=2}^{d_n-1}(-1)^{k+1}(D_n a_{N,n})(\bm x_0)\right\}\\
& =\sum_{n=1}^N\left\{\f{D_n a_{N,n}(\bm x_0)}{\la_n+\sum_{j=1}^m q_j p^{\al_j}}-\f{\la_n(D_n a_{N,n})(\bm x_0)}{(\la_n+\sum_{j=1}^m q_j p^{\al_j})^2}\right\}+O(p^{-2\al_1})\\
& =\sum_{n=1}^N\f{D_n a_{N,n}(\bm x_0)}{\la_n+\sum_{j=1}^m q_j p^{\al_j}}+O(p^{-2\al_1})\quad\mbox{for large }p>0.
\end{align*}
Plugging the expansion \eqref{eq-i1-denomi2} into the above equality and treating $J_N^2(p)$ in the same manner, we deduce
\begin{align}
p^{2\al_1-\be_\ell}(I_N^2(p)-J_N^2(p)) & =p^{\al_1-\be_\ell}\sum_{n=1}^N\left(\f{D_n a_{N,n}(\bm x_0)}{\sum_{j=1}^{\ell-1}q_j p^{\al_j-\al_1}}-\f{D_n b_{N,n}(\bm x_0)}{\ka^{-1}\sum_{j=1}^{\ell-1}q_j p^{\al_j-\al_1}}\right)\nonumber\\
& \quad\,-p^{\al_\ell-\be_\ell}(q_\ell+o(1))\sum_{n=1}^N\f{D_n a_{N,n}(\bm x_0)}{(q_1+o(1))^2}+(r_\ell+o(1))\sum_{n=1}^N\f{D_n b_{N,n}(\bm x_0)}{(r_1+o(1))^2}\nonumber\\
& \quad\,+O(p^{-\be_\ell})\sum_{n=1}^N\la_n D_n(a_{N,n}-b_{N,n})(\bm x_0)+O(p^{2\al_\ell-\al_1-\be_\ell})\sum_{n=1}^N D_n a_{N,n}(\bm x_0)\nonumber\\
& \quad\,+O(p^{\be_\ell-\al_1})\sum_{n=1}^N D_n b_{N,n}(\bm x_0)\nonumber\\
& =\f{p^{\al_1-\be_\ell}(1+o(1))}{q_1}\sum_{n=1}^N D_n(a_{N,n}-\ka\,b_{N,n})(\bm x_0)+\f{r_\ell+o(1)}{r_1^2}\sum_{n=1}^N\la_n b_{N,n}(\bm x_0)\nonumber\\
& \quad\,+o(1)\quad\mbox{for large }p>0.            \label{eq-ij2}
\end{align}
Collecting \eqref{eq-ij1}, \eqref{eq-ij2}, \eqref{eq-est-i3}, \eqref{eq-est-j3} and using \eqref{eq-lan} again, we obtain
\begin{align*}
O(p^{-\nu+2\al_1-\be_\ell}) & =p^{2\al_1-\be_\ell}\sum_{i=1}^3(I_N^i(p)-J_N^i(p))\\
& =\f{p^{\al_1-\be_\ell}(1+o(1))}{q_1}\sum_{n=1}^N(\la_n+D_n)(a_{N,n}-\ka\,b_{N,n})(\bm x_0)+\f{r_\ell+o(1)}{r_1^2}\sum_{n=1}^N(\la_n+D_n)b_{N,n}(\bm x_0)\\
& \quad\,+o(1)+O(p^{\al_1-\be_\ell})\left(\|a-a_N\|_{\cD(L^\ga)}+\|b-b_N\|_{\cD(L^\ga)}\right)\\
& =O(p^{\al_1-\be_\ell})\left(\sum_{n=1}^N L(a_{N,n}-\ka\,b_{N,n})(\bm x_0)+\|a-a_N\|_{\cD(L^\ga)}+\|b-b_N\|_{\cD(L^\ga)}\right)\\
& \quad\,+\f{r_\ell+o(1)}{r_1^2}\sum_{n=1}^N L b_{N,n}(\bm x_0)+o(1)\\
& =O(p^{\al_1-\be_\ell})\left(L(a_N-\ka\,b_N)(\bm x_0)+\|a-a_N\|_{\cD(L^\ga)}+\|b-b_N\|_{\cD(L^\ga)}\right)\\
& \quad\,+\f{r_\ell+o(1)}{r_1^2}L b_N(\bm x_0)+o(1)\quad\mbox{for large }p>0.
\end{align*}
By the assumption $\nu>2\min\{\al_1,\be_1\}=2\al_1$, we know $-\nu+2\al_1-\be_\ell<0$ and thus
\begin{equation}\label{eq-est3}
O(p^{\al_1-\be_\ell})\left(L(a_N-\ka\,b_N)(\bm x_0)+\|a-a_N\|_{\cD(L^\ga)}+\|b-b_N\|_{\cD(L^\ga)}\right)+\f{r_\ell+o(1)}{r_1^2}L b_N(\bm x_0)=o(1)
\end{equation}
for large $p>0$. By the definition of $\ka$, we pass $N\to\infty$ in \eqref{eq-est3} and employ \eqref{eq-cov1}--\eqref{eq-cov2} again to deduce
\[
\f{r_\ell+o(1)}{r_1^2}L b(\bm x_0)=o(1)\quad\mbox{for large }p>0.
\]
Consequently, passing $p\to\infty$ yields $L b(\bm x_0)=0$, which contradicts the assumption $L b(\bm x_0)\ne0$. In an identically parallel manner, we can exclude the possibility of $\al_\ell>\be_\ell$ and conclude $\al_\ell=\be_\ell$.

Now it suffices to verify $q_\ell/r_\ell=\ka$. Similarly to the argument at the end of Step 2, it follows immediately from $\al_\ell=\be_\ell$ that now \eqref{eq-ij1} and \eqref{eq-ij2} become
\begin{align*}
p^{2\al_1-\be_\ell}(I_N^1(p)-J_N^1(p)) & =\f{p^{\al_1-\be_\ell}(1+o(1))}{q_1}\sum_{n=1}^N\la_n(a_{N,n}-\ka\,b_{N,n})(\bm x_0)-\f{q_\ell+o(1)}{q_1^2}\sum_{n=1}^N\la_n a_{N,n}(\bm x_0)\nonumber\\
& \quad\,+\f{r_\ell+o(1)}{r_1^2}\sum_{n=1}^N\la_n b_{N,n}(\bm x_0)+o(1),\\
p^{2\al_1-\be_\ell}(I_N^2(p)-J_N^2(p)) & =\f{p^{\al_1-\be_\ell}(1+o(1))}{q_1}\sum_{n=1}^N D_n(a_{N,n}-\ka\,b_{N,n})(\bm x_0)-\f{q_\ell+o(1)}{q_1^2}\sum_{n=1}^N\la_n a_{N,n}(\bm x_0)\nonumber\\
& \quad\,+\f{r_\ell+o(1)}{r_1^2}\sum_{n=1}^N\la_n b_{N,n}(\bm x_0)+o(1).
\end{align*}
Repeating the same calculation as before, we arrive at
\begin{align*}
o(1) & =O(p^{\al_1-\al_\ell})\left(L(a_N-\ka\,b_N)(\bm x_0)+\|a-a_N\|_{\cD(L^\ga)}+\|b-b_N\|_{\cD(L^\ga)}\right)\\
& \quad\,-\f{q_\ell+o(1)}{q_1^2}L a_N(\bm x_0)+\f{r_\ell+o(1)}{r_1^2}L a_N(\bm x_0)\quad\mbox{for large } p>0.
\end{align*}
Again passing first $N\to\infty$ and then $p\to\infty$, we conclude
\[
-\f{q_\ell}{q_1^2}L a(\bm x_0)+\f{r_\ell}{r_1^2}L b(\bm x_0)=0
\]
or equivalently $q_\ell/r_\ell=\ka$. This completes the proof of \eqref{eq-induct} for $j=\ell$. By the inductive assumption, eventually \eqref{eq-induct} holds for all $j=1,2,\ldots,\min\{m,m'\}$.

{\bf Step 4 } At last, it remains to show $m=m'$ again by contradiction, that is, we assume $m<m'$ without loss of generality. As those in the previous steps, our strategy is again multiplying both sides of \eqref{eq-est1} by $p^{2\al_1-\be_{m+1}}$ to deduce a contradiction.

Notice that now $J_N^1(p)$ and $J_N^2(p)$ are represented by \eqref{eq-j12} with $\ell=m+1$. Treating the denominators in $I_N^1(p)$ and $J_N^1(p)$ similarly as that in \eqref{eq-i1-denomi2} yields
\begin{align*}
\f1{\la_n+\sum_{j=1}^m q_j p^{\al_j}} & =\f{p^{-\al_1}}{\sum_{j=1}^m q_j p^{\al_j-\al_1}}+\la_n O(p^{-2\al_1}),\\
\f1{\la_n+\sum_{j=1}^{m'}r_j p^{\be_j}} & =\f{\ka\,p^{-\al_1}}{\sum_{j=1}^m q_j p^{\al_j-\al_1}}-\f{p^{2\al_1-\be_{m+1}}(r_{m+1}+o(1))}{(r_1+o(1))^2}+\la_n O(p^{-2\al_1})+O(p^{2\be_{m+1}-3\al_1})
\end{align*}
for large $p>0$. Then
\begin{align}
p^{2\al_1-\be_{m+1}}(I_N^1(p)-J_N^1(p)) & =\f{p^{\al_1-\be_{m+1}}(1+o(1))}{q_1}\sum_{n=1}^N\la_n(a_{N,n}-\ka\,b_{N,n})(\bm x_0)\nonumber\\
& \quad\,+\f{r_{m+1}+o(1)}{r_1^2}\sum_{n=1}^N\la_n b_{N,n}(\bm x_0)+o(1)\quad\mbox{for large }p.                         \label{eq-ij1'}
\end{align}
Treating $I_N^2(p)$ and $J_N^2(p)$ similarly to 
that in Step 3 and expanding the involved denominators as above, we have
\begin{align}
p^{2\al_1-\be_{m+1}}(I_N^2(p)-J_N^2(p)) & =\f{p^{\al_1-\be_{m+1}}(1+o(1))}{q_1}\sum_{n=1}^N D_n(a_{N,n}-\ka\,b_{N,n})(\bm x_0)\nonumber\\
& \quad\,+\f{r_{m+1}+o(1)}{r_1^2}\sum_{n=1}^N D_n b_{N,n}(\bm x_0)+o(1)\quad\mbox{for large }p.\label{eq-ij2'}
\end{align}
Collecting \eqref{eq-ij1'}, \eqref{eq-ij2'}, \eqref{eq-est-i3}, \eqref{eq-est-j3} and repeating the same calculation used in Step 3, we reach \eqref{eq-est3} with $\ell=m+1$. Performing the same limiting process as before, we conclude again the contradiction $L b(\bm x_0)=0$. This completes the proof of $m=m'$ and thus finalizes the proof of part (ii).
\end{proof}

Now we turn to the proof of Theorem \ref{thm-inexact1}, which relies on the following key lemma concerning the temporal component $\rho(t)$ of the inhomogeneous terms in \eqref{eq-ibvp-u} and \eqref{eq-ibvp-v}.

\begin{lem}\label{lem-rho}
Let $\rho\in L^1(\BR_+)$ satisfy \eqref{eq-rho}. Then its Laplace transform $\wh\rho(p)$ satisfies $\wh\rho(p)\sim p^{-\mu-1}$ for large $p>0$.
\end{lem}

\begin{proof}
According to \eqref{eq-rho}, there exist a small constant $\ve>0$ and 
constants $C_1,C_2\in\BR$ satisfying $C_1<C_2$ and $C_1C_2>0$ such that
\begin{equation}\label{eq-rho1}
C_1t^\mu\le\rho(t)\le C_2t^\mu\quad\mbox{a.e. }t\in(0,\ve).
\end{equation}
By the definition of the Laplace transform, we decompose $\wh\rho(p)$ as
\[
\wh\rho(p)=\left(\int_0^\ve+\int_\ve^\infty\right)\e^{-p t}\rho(t)\,\rd t=:I_1(p)+I_2(p),
\]
where
\[
|I_2(p)|\le\int_\ve^\infty\e^{-p t}|\rho(t)|\,\rd t\le\e^{-\ve p}\|\rho\|_{L^1(\BR_+)}=o(p^{-\mu-1})\quad\mbox{for large }p>0.
\]
For $I_1(p)$, it follows from \eqref{eq-rho1} and
\[
\int_0^\ve\e^{-p t}t^\mu\,\rd t=\f1{p^{\mu+1}}\int_0^{\ve p}\e^{-s}s^\mu\,\rd s\sim p^{-\mu-1}\quad\mbox{for large }p>0
\]
that $I_1(p)\sim p^{-\mu-1}$ for large $p>0$, which completes the proof.
\end{proof}

\begin{proof}[Proof of Theorem $\ref{thm-inexact1}$]
Let $a=b\equiv0$ in $\Om$ in \eqref{eq-ibvp-u} and \eqref{eq-ibvp-v}. Similarly to Step 1 in the proof of Theorem \ref{thm-inexact}, it is not difficult to conclude that the difference between $\wh u(\bm x_0,p)$ and $\wh v(\bm x_0,p)$ still satisfies \eqref{eq-est0} for large $p>0$.

Next, we take the Laplace transform in \eqref{eq-ibvp-u} and \eqref{eq-ibvp-v} and substitute $\bm x=\bm x_0$ to deduce
\[
\wh u(\bm x_0,p)-\wh v(\bm x_0,p)=\wh\rho(p)\left(w(\bm x_0,p)-y(\bm x_0,p)\right)\quad\mbox{for large }p>0,
\]
where
\begin{equation}\label{eq-exp-wy}
w(\,\cdot\,,p):=\left(L+\sum_{j=1}^m q_j p^{\al_j}\right)^{-1}f,\quad y(\,\cdot\,,p):=\left(L+\sum_{j=1}^{m'}r_j p^{\be_j}\right)^{-1}g.
\end{equation}
Then we employ \eqref{eq-est0} and Lemma \ref{lem-rho} to obtain
\begin{equation}\label{eq-est-wy}
w(\bm x_0,p)-y(\bm x_0,p)=O(p^{\mu-\nu})\quad\mbox{for large }p>0.
\end{equation}
Arguing similarly as before, we utilize Lemma \ref{lem-total} to pick two sequences $f_N=\sum_{n=1}^N f_{N,n}$ and $g_N=\sum_{n=1}^N g_{N,n}$ ($N\in\BN$) such that
\[
f_{N,n},g_{N,n}\in P_n L^2(\Om),\quad\lim_{N\to\infty}\|L^\ga(f-f_N)\|_{L^2(\Om)}=\lim_{N\to\infty}\|L^\ga(g-g_N)\|_{L^2(\Om)}=0.
\]
Then for any fixed $N\in\BN$, we decompose $f=f_N+(f-f_N)$ and employ the expansion \eqref{eq-Laurant} to rewrite \eqref{eq-exp-wy} as
\begin{align*}
w(\,\cdot\,,p) & =\left(L+\sum_{j=1}^m q_j p^{\al_j}\right)^{-1}\left\{\sum_{n=1}^N f_{N,n}+(f-f_N)\right\}\\
& =\sum_{n=1}^N\left(L+\sum_{j=1}^m q_j p^{\al_j}\right)^{-1}P_n f_{N,n}+\left(L+\sum_{j=1}^m q_j p^{\al_j}\right)^{-1}(f-f_N)\\
& =\sum_{n=1}^N\f{f_{N,n}}{\la_n+\sum_{j=1}^m q_j p^{\al_j}}+\sum_{n=1}^N\sum_{k=1}^{d_n-1}\f{(-1)^k D_n^k f_{N,n}}{(\la_n+\sum_{j=1}^m q_j p^{\al_j})^{k+1}}+\left(L+\sum_{j=1}^m q_j p^{\al_j}\right)^{-1}(f-f_N)
\end{align*}
and similarly
\[
y(\,\cdot\,,p)=\sum_{n=1}^N\f{g_{N,n}}{\la_n+\sum_{j=1}^{m'}r_j p^{\be_j}}+\sum_{n=1}^N\sum_{k=1}^{d_n-1}\f{(-1)^k D_n^k g_{N,n}}{(\la_n+\sum_{j=1}^{m'}r_j p^{\be_j})^{k+1}}+\left(L+\sum_{j=1}^{m'}r_j p^{\be_j}\right)^{-1}(g-g_N).
\]
Then taking $\bm x=\bm x_0$ in the above expressions and substituting them into \eqref{eq-est-wy} yield
\[
\sum_{i=1}^3(I_N^i(p)-J_N^i(p))=O(p^{\mu-\nu})\quad\mbox{for large }p>0,
\]
where
\begin{alignat*}{2}
I_N^1(p) & :=\sum_{n=1}^N\f{f_{N,n}(\bm x_0)}{\la_n+\sum_{j=1}^m q_j p^{\al_j}}, & \quad J_N^1(p) & :=\sum_{n=1}^N\f{g_{N,n}(\bm x_0)}{\la_n+\sum_{j=1}^{m'}r_j p^{\be_j}},\\
I_N^2(p) & :=\sum_{j=1}^m q_j p^{\al_j}\sum_{n=1}^N\sum_{k=1}^{d_n-1}\f{(-1)^k(D_n^k f_{N,n})(\bm x_0)}{(\la_n+\sum_{j=1}^m q_j p^{\al_j})^{k+1}}, & \quad J_N^2(p) & :=\sum_{j=1}^{m'}r_j p^{\be_j}\sum_{n=1}^N\sum_{k=1}^{d_n-1}\f{(-1)^k(D_n^k g_{N,n})(\bm x_0)}{(\la_n+\sum_{j=1}^{m'}r_j p^{\be_j})^{k+1}},\\
I_N^3(p) & :=\left(\left(L+\sum_{j=1}^m q_j p^{\al_j}\right)^{-1}(f-f_N)\right)(\bm x_0), & \quad J_N^3(p) & :=\left(\left(L+\sum_{j=1}^{m'}r_j p^{\be_j}\right)^{-1}(g-g_N)\right)(\bm x_0).
\end{alignat*}
The remaining part of the proof turns out to be essentially the same as that of Theorem \ref{thm-inexact} and we refrain from repeating the details.
\end{proof}

\begin{proof}[Proof of Theorem $\ref{thm-unique}$]
We only provide the proof for part (1) because that of part (2) is mostly parallel. Then throughout this proof, we set $f=g\equiv0$ in $\Om$ in \eqref{eq-ibvp-u} and \eqref{eq-ibvp-v}. Notice that thanks to the symmetry of $L$, we have \eqref{eq-symm} and
\begin{equation}\label{eq-symm1}
h=\sum_{n=1}^\infty P_n h,\ \forall\,h\in L^2(\Om),\quad(L-z)^{-1}P_n=\f{P_n}{\la-z},\ z\not\in\si(L).
\end{equation}

First we assume \eqref{eq-exact} and show the ``only if'' part. According to Theorem \ref{thm-inexact}, we have already reached \eqref{eq-unique2}. Together with \eqref{eq-symm1}, we can rewrite \eqref{eq-exp-u} and \eqref{eq-exp-v} as
\begin{equation}\label{eq-iff}
\begin{aligned}
\wh u(\,\cdot\,,p) & =\f1p\sum_{j=1}^m q_j p^{\al_j}\sum_{n=1}^\infty\f{P_n a}{\la_n+\sum_{j=1}^m q_j p^{\al_j}},\\
\wh v(\,\cdot\,,p) & =\f1p\sum_{j=1}^m q_j p^{\al_j}\sum_{n=1}^\infty\f{P_n b}{\ka\la_n+\sum_{j=1}^m q_j p^{\al_j}}
\end{aligned}\quad\mbox{for large }p>0.
\end{equation}
On the other hand, owing to the time-analyticity of the solutions to \eqref{eq-ibvp-u} and \eqref{eq-ibvp-v}, we see that \eqref{eq-exact} implies $u(\bm x_0,\,\cdot\,)=v(\bm x_0,\,\cdot\,)$ in $\BR_+$. Therefore, taking the Laplace transform results in $\wh u(\bm x_0,p)=\wh v(\bm x_0,p)$ for large $p>0$ and hence
\[
z\sum_{n=1}^\infty\f{P_n a(\bm x_0)}{\la_n+z}=z\sum_{n=1}^\infty\f{P_n b(\bm x_0)}{\ka\la_n+z}\quad\mbox{for large }z,\quad z:=\sum_{j=1}^m q_j p^{\al_j}.
\]
Then we can analytically extend both sides of the above equality in $z\in\BC$ to obtain
\begin{equation}\label{eq-cauchy}
\sum_{n=1}^\infty\f{P_n a(\bm x_0)}{\la_n+z}=\sum_{n=1}^\infty\f{P_n b(\bm x_0)}{\ka\la_n+z},\quad\sum_{n=1}^\infty\f{P_n L a(\bm x_0)}{\la_n+z}=\sum_{n=1}^\infty\f{\ka\,P_n L b(\bm x_0)}{\ka\la_n+z},\quad z\in\BC\setminus\La,
\end{equation}
where
\[
\La:=(-\si(L))\cup(-\ka\,\si(L)),\quad\eta\,\si(L):=\{\eta\la_n\mid\la_n\in\si(L)\},\quad\eta\in\BR.
\]
Here we used \eqref{eq-a=b} and the fact that $\la_n L=L P_n$.

Now we concentrate on $\ka$ and consider the cases of $\ka=1$ and $\ka\ne1$ separately.

{\bf Case 1 } If $\ka=1$, then $\La=-\si(L)$ and for each $n\in\BN$, we can take a sufficiently small circle $\ga_n$ enclosing $-\la_n$ and excluding $-\si(L)\setminus\{-\la_n\}$. Integrating \eqref{eq-cauchy} on $\ga_n$ with respect to $z$, we can employ Cauchy's integral theorem to conclude \eqref{eq-pab1}. In view of Remark \ref{rem-ka1}, this completes the proof of \eqref{eq-pab} for $\ka=1$.

{\bf Case 2 } Now let us assume $\ka\ne1$. To obtain \eqref{eq-iff1}, it suffices to show $\ka\in\Si$ by contradiction. If $\ka\not\in\Si$, then it follows from the definition of $\Si$ that $\si(L)\cap\ka\,\si(L)=\emptyset$. Thus for each $n\in\BN$, we can take the same circle $\ga_n$ as that in Case 1 enclosing $-\la_n$ and excluding $\La\setminus\{-\la_n\}$. Then integrating \eqref{eq-cauchy} on $\ga_n$ again yields $P_n L a(\bm x_0)=0$ ($\forall\,n\in\BN$) and hence
\[
L a(\bm x_0)=\sum_{n=1}^\infty P_n L a(\bm x_0)=0,
\]
which contradicts the assumption $L a(\bm x_0)\ne0$. This verifies $\ka\in\Si$ and also \eqref{eq-iff1}.

Now we turn to demonstrating \eqref{eq-pab} for $\ka\ne1$. Again by the definition of $\Si$, we know that $\si(L)\cap\ka\,\si(L)\ne\emptyset$ and the index sets $M_\ka,M_\ka'$ defined by \eqref{eq-def-mm} are obviously nonempty. Meanwhile, by the strict monotonicity ordering of $\si(L)$, it is readily seen that for each $n\in M_\ka$, there exists a unique $n'\in M_\ka'$ such that $\la_n=\ka\,\la_{n'}$ and vise versa. Then the bijection $\te_\ka:M_\ka\longrightarrow M_\ka'$ can be defined by $\te_\ka(n)=n'$.

Next, for each $n\not\in M_\ka$, we have $\la_n\not\in\ka\,\si(L)$. Then we can take the same circle $\ga_n$ as before and integrating \eqref{eq-cauchy} on $\ga_n$ to conclude
\[
P_n a(\bm x_0)=P_n L a(\bm x_0)=0,\quad\forall\,n\not\in M_\ka.
\]
Similarly, we have
\[
P_n b(\bm x_0)=P_n L b(\bm x_0)=0,\quad\forall\,n\not\in M_\ka'.
\]
Finally, for each $n\in M_\ka$, we can take a circle $\ga_n$ enclosing $\la_n=\ka\la_{\te_\ka(n)}$ and excluding $\La\setminus\{-\la_n\}$. Then again integrating \eqref{eq-cauchy} on $\ga_n$ gives
\[
P_n a(\bm x_0)=P_{\te_ka(n)} b(\bm x_0),\quad P_n L a(\bm x_0)=\ka\,P_{\te_\ka(n)} L b(\bm x_0),\quad\forall\,n\in M_\ka.
\]
Collecting the above three cases completes the proof of \eqref{eq-pab} for $\ka\ne1$.

Now it remains to show the ``if'' part. Taking $\bm x=\bm x_0$ in \eqref{eq-iff} and employing conditions \eqref{eq-def-mm}--\eqref{eq-pab}, we obtain
\begin{align*}
\wh u(\bm x_0,p) & =\f z p\sum_{n=1}^\infty\f{P_n a(\bm x_0)}{\la_n+z}=\f z p\sum_{n\in M_\ka}\f{P_n a(\bm x_0)}{\la_n+z}=\f z p\sum_{n\in M_\ka}\f{P_{\te_\ka(n)}b(\bm x_0)}{\la_n+z}\\
& =\f z p\sum_{n\in M_\ka'}\f{P_n b(\bm x_0)}{\ka\la_n+z}=\f z p\sum_{n=1}^\infty\f{P_n b(\bm x_0)}{\ka\la_n+z}=\wh v(\bm x_0,p)\quad\mbox{for large }p,
\end{align*}
where we abbreviated $z:=\sum_{j=1}^m q_j p^{\al_j}$. Then the coincidence of $u$ and $v$ at $\{\bm x_0\}\times\BR_+$ follows immediately from the uniqueness of the Laplace transform.
\end{proof}

\begin{proof}[Proof of Corollary $\ref{coro-unique}$]
Similarly to the proof of Theorem \ref{thm-unique}, we only deal with part (1) and omit the proof of part (2). Then it suffices to show the equivalence of \eqref{eq-pab} and \eqref{eq-rel-ab}.

By the assumption \eqref{eq-omega} on the choice of $\bm x_0$, there holds $\vp_n(\bm x_0)\ne0$ for all $n\in\BN$. On the other hand, since all eigenvalues of $L$ are assumed to be simple, it reveals that
\[
P_n h=(h,\vp_n)\vp_n,\quad P_n L h=\la_n(h,\vp_n)\vp_n,\quad\forall\,h\in\cD(L).
\]
Therefore, noticing that $\la_n=\ka\la_{\te_\ka(n)}$ for $n\in M_\ka$, we immediately see that \eqref{eq-pab} is equivalent to
\[
\begin{cases}
(a,\vp_n)\vp_n(\bm x_0)=(b,\vp_{\te_\ka(n)})\vp_{\te_\ka(n)}(\bm x_0), & n\in M_\ka,\\
(a,\vp_n)=0, & n\not\in M_\ka,\\
(b,\vp_n)=0, & n\not\in M_\ka'.
\end{cases}
\]
Substituting the above relation into the Fourier expansions of $a,b$, we immediately obtain \eqref{eq-rel-ab}. The opposite side can also be verified similarly.

Especially if $\ka=1$, it turns out that \eqref{eq-pab} simply equals
\[
(a,\vp_n)=(b,\vp_n),\quad\forall\,n\in\BN
\]
or equivalently $a=b$ in $\Om$. The proof of Corollary \ref{coro-unique} is completed.
\end{proof}

\section{Concluding Remarks}\label{sec-rem}

As information of solutions for our inverse problem,
we assume only \eqref{eq-inexact}, 
which means the coincidence of principal terms of 
asymptotic expansions of solutions as $t\downarrow0$.
On the other hand, after taking Laplace transform of the solutions with respect to $t$, we only take advantage of the asymptotic behavior of $\wh u(\bm x_0,p)-\wh v(\bm x_0,p)$ for large $p>0$. This reflects the dominating effect of 
fractional orders and corresponding parameters to the solution. Simultaneously, it also suggests that the asymptotic behavior of solutions possesses abundant information about the equations.

The starting point of the proofs relies on the Laplace transforms of solutions, whose existence was guaranteed by the key estimate \eqref{eq-est-u} in Lemma \ref{lem-fp}. Nevertheless, especially for $\si(L)\subset\BR_+$, the estimate \eqref{eq-est-u} is not at all sharp for large $t$. The sharp decay rate with a non-symmetric $L$ keeps open even in the single-term case of \eqref{eq-ibvp-u}, which deserves further investigation.

By reviewing the proof, we immediately see that our arguments also work for multi-term time-fractional wave equations, that is, \eqref{eq-ibvp-u} and \eqref{eq-ibvp-v} with
\begin{gather*}
2>\al_1>\al_2>\cdots>\al_m>0,\quad2>\be_1>\be_2>\cdots>\be_{m'}>0,\\
\pa_t u=0\quad\mbox{if }\al_1>1,\quad\pa_t v=0\quad\mbox{if }\be_1>1\quad\mbox{in }\Om\times\{0\}.
\end{gather*}
However, since the forward theory for multi-term time-fractional wave equations (especially with a non-symmetric elliptic part) is not well established for the moment, we postpone this generalization as a future work. Meanwhile, other possible future topics include the case of $\bm x$-dependent parameters $q_j(\bm x),r_j(\bm x)$.

Recently researches of wave equations with time-fractional damping terms,
are developed. Such fractional systems and inverse problems 
are important also from physical points of view, and we can refer to 
Kaltenbacher and Rundell \cite{KR22}. We can expect that 
our arguments may be available, but this should be a future work.

Our main purpose is the determination of scalar parameters $m, \alpha_j,
q_j$ in a multi-term time-fractional diffusion equation by some minimal 
information of solutions, as which we adopt the asymptotic behavior 
near $t=0$.  In general, it is difficult to prove the uniqueness in 
determining spatially varying initial data or source terms by such sparse data.

We can discuss weakened information of \eqref{eq-inexact}:
\begin{gather}
\mbox{There exist a sequence $\{t_m\}_{m=1}^\infty\subset\BR_+$ 
and a constant $C>0$ such that}\nonumber\\
\lim_{m\to\infty}t_m=0\quad\mbox{and}\quad
|u(\bm x_0,t_m)-v(\bm x_0,t_m)|\le C\,t^{\wt\nu}_m, \quad m\in \BN,\label{eq-5.1}
\end{gather}
where $\wt\nu>\min\{\al_1,\be_1\}$.
From a practical viewpoint, $\{t_m\}$ is a set of 
sampling times at a spatial monitoring point $\bm x_0\in\Om$.
We can expect that \eqref{eq-5.1} is equivalent to \eqref{eq-inexact}, so that 
replacing \eqref{eq-inexact} by the weakened condition \eqref{eq-5.1}, we may obtain the same conlclusion as in Theorem \ref{thm-inexact}.

The proof in the general case is postponed to a future work, and here we 
sketch in the case $m=1$ with smooth $a,b\in\cD(L^\ga)$ where $\ga> 0$ is sufficiently large. We assume that $(L a)(\bm x_0)\ne0$, $(L b)(\bm x_0)\ne0$
and we write $\al=\al_1$, $\be=\be_1$ for simplicity. Then 
\eqref{eq-5.1} holds with some $\wt\nu>\min\{\al,\be\}$ if
and only if \eqref{eq-inexact} holds with some $\nu>\min\{\al,\be\}$.

\begin{proof}
It is trivial that \eqref{eq-inexact} implies \eqref{eq-5.1}.
Assuming \eqref{eq-5.1}, we see that 
\[
\begin{aligned}
u(\bm x_0,t) & =\sumn E_{\al,1}\left(-\f{\la_n}{q_1}t^\al\right)(P_n a)(\bm x_0),\\
v(\bm x_0,t) & =\sumn E_{\be,1}\left(-\f{\la_n}{r_1}t^\al\right)(P_n b)(\bm x_0),
\end{aligned}\quad t>0.
\]
For $a,b\in\cD(L^\ga)$ with large $\ga>0$, we can calculate
\begin{align*}
u(\bm x_0,t) & =\sumn E_{\al,1}\left(-\f{\la_n}{q_1}t^\al\right)(P_n a)(\bm x_0)\\
& =\sumn\left\{\left(1-\f{\la_n t^\al}{q_1\Ga(\al+1)}\right)P_n a(\bm x_0)
+\sum_{k=2}^\infty\f{(-\f{\la_n}{q_1}t^\al)^k}{\Ga(\al k+1)}P_n a(\bm x_0) \right\}\\
& =\sumn(P_n a)(\bm x_0)-\f{t^\al}{q_1\Ga(\al+1)}\sumn\la_n(P_n a)(\bm x_0)\\
& \quad\,+\sumn\sum_{j=0}^\infty\f1{\Ga(\al j+2\al+1)}
\left(-\f{\la_n}{q_1}t^\al\right)^2\left(-\f{\la_n}{q_1}t^\al\right)^j(P_n a)(\bm x_0)\\
& =a(\bm x_0)-\f{t^\al}{q_1\Ga(\al+1)}(L a)(\bm x_0)+\f{t^{2\al}}{q_1^2}
\sumn E_{\al,2\al+1}\left(-\f{\la_n}{q_1}t^\al\right)P_n(L^2a)(\bm x_0).
\end{align*}
Since $a$ is sufficiently smooth, we can verify by the estimate of 
$E_{\al,2\al+1}$ (e.g., \cite[Theorem 1.6 (p.35)]{P99}) that
$$
\sumn\left|E_{\al,2\al+1}\left(-\f{\la_n}{q_1}t^\al\right)\right||P_n(L^2a)(\bm x_0)|
\le\sumn\f C{1+\f{\la_n}{q_1}t^\al}|P_n(L^2a)(\bm x_0)|<\infty.
$$
Hence, we obtain
\begin{equation}\label{eq-5.2}
u(\bm x_0,t)=a(\bm x_0)-\f{t^\al}{q_1\Ga(\al+1)}(L a)(\bm x_0)+O(t^{2\al})
\quad\mbox{as }t\downarrow0.
\end{equation}
Similarly, we can prove
\begin{equation}\label{eq-5.3}
v(\bm x_0,t)=b(\bm x_0)-\f{t^\be}{r_1\Ga(\be+1)}(L b)(\bm x_0)+O(t^{2\be})
\quad\mbox{as }t\downarrow0.
\end{equation}
Therefore, \eqref{eq-5.1} yields
\[
\left|a(\bm x_0)-b(\bm x_0)-\f{(L a)(\bm x_0)}{q_1\Ga(\al+1)}t_m^\al
+\f{(L b)(\bm x_0)}{r_1\Ga(\be+1)}t_m^\be+O(t_m^{2\al})+O(t_m^{2\be})\right|
\le C\,t_m^{\wt\nu},\quad m\in\BN.
\]
Here $\wt\nu>\min\{\al,\be\}$. Letting $m\to\infty$, we see \eqref{eq-a=b}.
Assuming $\al\le\be$ without loss of generality, we see $\wt\nu>\al$.
Dividing the above inequality by $t_m^\al$, we have
\begin{equation}\label{eq-5.5}
\left|-\f{(L a)(\bm x_0)}{q_1\Ga(\al+1)}
+\f{(L b)(\bm x_0)}{r_1\Ga(\be+1)}t_m^{\be-\al}
+O(t_m^\al)+O(t_m^{2\be-\al})\right|
\le C\,t_m^{\wt\nu-\al},\quad m\in\BN.
\end{equation}
Letting $m\to\infty$, since $O(t_m^{2\be-\al})=o(1)$ and $t_m^{\wt\nu-\al}=o(1)$ as 
$m\to\infty$, we reach 
$$
\f{(L a)(\bm x_0)}{q_1\Ga(\al+1)}
=\lim_{m\to\infty}t_m^{\be-\al}\left(\f{(L b)(\bm x_0)}{r_1\Ga(\be+1)}\right).
$$
By the assumption $(L a)(\bm x_0)\ne0$ and $(L b)(\bm x_0)\ne0$, we obtain
$$
\lim_{m\to\infty}t_m^{\be-\al}\left(\f{(L b)(\bm x_0)}{r_1\Ga(\be+1)}\right)\ne0,
$$
that is, $\lim_{m\to\infty}t_m^{\be-\al}\ne0$, which implies
$\be=\al$. Letting $m\to\infty$ in \eqref{eq-5.5} with $\al=\be$, we see that 
$$
\f{(L b)(\bm x_0)}{r_1\Ga(\be+1)}=\f{(L a)(\bm x_0)}{q_1\Ga(\al+1)}.
$$
Therefore, by \eqref{eq-5.2}--\eqref{eq-5.5}, we obtain
$$
|u(\bm x_0,t)-v(\bm x_0,t)|= O(t^{2\al})\quad\mbox{as }t\downarrow0.
$$
Hence, \eqref{eq-inexact} is satisfied with $\nu:=2\al>\min\{\al,\be\}$.
Thus \eqref{eq-inexact} and \eqref{eq-5.1} provide equivalent information.
\end{proof}

Although the single point observation is sufficient for the uniqueness of 
Problem \ref{prob-pip} theoretically, it seems meaningful to develop numerical methods with multiple observation points in order to improve the accuracy for 
the reconstruction. 
If the corresponding rank condition (see \cite{S75}) is also satisfied 
in this case, it is even possible to uniquely identify the initial value $a$ or the spatial component $f$ of the source term in \eqref{eq-ibvp-u}, whose numerical reconstruction may also be interesting.

Finally, for Problem \ref{prob-pip}, we can adopt information of the asymptotic data
as $t\to\infty$, but we postpone details to another future work.
 

\section*{Acknowledgments}

Y.\! Liu is supported by Grant-in-Aid for Early-Career Scientists 20K14355 and 22K13954, Japan Society for the Promotion of Science (JSPS).
M.\! Yamamoto is supported by Grant-in-Aid for Scientific Research (A) 20H00117 and Grant-in-Aid for Challenging Research (Pioneering) 21K18142, JSPS.

\bibliographystyle{unsrt}

\end{document}